\documentclass[a4paper,12pt,twoside]{article}
\usepackage{indentfirst}
\usepackage{amsmath,amssymb,mathrsfs}

\newtheorem{theorem}{Theorem}[section]

\newtheorem{lemma}[theorem]{Lemma}
\newtheorem{proposition}[theorem]{Proposition}
\newtheorem{definition}[theorem]{Definition}

\newenvironment{proof}[1][Proof]{{\par \emph{#1.}}\;}{\hfill $\Box$\par}
\numberwithin{equation}{section}
\allowdisplaybreaks
\newcommand{\norm}[1]{\|#1\|}
\newcommand{\abs}[1]{\left|#1\right|}
\newcommand{\set}[1]{\left\{#1\right\}}
\newcommand{\Real}{\mathbb R}
\newcommand{\NN}{\mathbb N}
\newcommand{\Z}{\mathbb Z}

\newcommand{\eps}{\varepsilon}
\newcommand{\lam}{\Lambda}

\newcommand{\Sz}{\mathscr{S}}
\newcommand{\F}{\mathscr{F}}
\newcommand{\C}{\mathcal{C}}

\newcommand{\ls}{\leqslant}
\newcommand{\gs}{\geqslant}

\newcommand{\supp}{\text{{\,\rm supp}\,}}
\newcommand{\dive}{\text{\,\rm div}}
\newcommand{\divp}{\text{\,\rm div}^\perp}

\newcommand{\vu}{\mathbf{u}}
\newcommand{\inc}{d}

\newcommand{\vv}{\mathbf{v}}

\newcommand{\hb}{\bar{\h}_0}
\newcommand{\dk}[1][k]{\triangle_{#1}}

\newcommand{\te}[1][]{\tilde{#1}}
\newcommand{\be}[1][]{\dot{B}_{2,1}^{#1}}
\newcommand{\hbe}[2][s-1]{\tilde{B}_{2,1}^{#1,#2}}

\newcommand{\h}{h} 
\newcommand{\f}{\mathrm{f}} 
\newcommand{\g}{\mathrm{g}} 
\newcommand{\propref}[1]{\hyperref[#1]{Proposition~\ref{#1}}}
\newcommand{\thmref}[1]{\hyperref[#1]{Theorem~\ref{#1}}}
\newcommand{\lemref}[1]{\hyperref[#1]{Lemma~\ref{#1}}}
\newcommand{\secref}[1]{\hyperref[#1]{Section~\ref{#1}}}
\begin{document}

\title{\bf Cauchy problem for viscous rotating shallow water equations}

\pagestyle{myheadings}
\markboth{C.C.Hao, L.Hsiao, H.-L.Li}{Viscous Rotating Shallow Water Equations}
\author{Chengchun Hao${}^{1,}$, Ling Hsiao${}^1$, Hai-Liang Li${}^2$\\
\parbox{\textwidth}{\begin{center}
\footnotesize \it ${}^1$Institute of Mathematics, Academy of Mathematics \& Systems Science, CAS,\\
 Beijing 100190, P. R. China\\
${}^2$Department of Mathematics,
 Capital Normal University,
 Beijing 100037, P.R.China
\end{center}}
}

\date{}

\maketitle

\begin{abstract}
We consider the Cacuhy problem for a viscous compressible rotating
shallow water system with a third-order surface-tension term
involved, derived recently in the modelling of motions for shallow
water with free surface in a rotating sub-domain~\cite{Mar07}. The
global existence of the solution in the space of Besov type is shown for
initial data close to a constant equilibrium state away from the vacuum.
Unlike the previous analysis about the compressible fluid model
without coriolis forces, see for instance \cite{Dan00,Has08}, the
rotating effect causes a coupling between two parts of Hodge's
decomposition of the velocity vector field, and additional regularity is
required in order to carry out the Friedrichs' regularization and
compactness arguments.
\end{abstract}


\section{Introduction}
The nonlinear shallow water equation is used to model the motion of
a shallow layer of homogeneous incompressible fluid in a three
dimensional rotating sub-domain and, in particular, to simulate the
vertical average dynamics of the fluid in terms of the horizontal
velocity and depth variation. In general, it is modeled by the three
dimensional incompressible Navier-Stokes-Coriolis system in a
rotating sub-domain of $\Real^3 $ together with a (nonlinear)
free moving surface boundary condition for which the stress tension is evolved
at the air-fluid interface from above and the Navier boundary
condition of wall-law type holds at the bottom.  Under a large-scale
assumption and hydrostatic approximation, the nonlinear shallow
water equation has been derived recently in~\cite{GerbeauPerthame01,Mar07}.
Usually, the nonlinear shallow water equations take the  following
form of compressible Navier-Stokes equations
\begin{equation}\label{0sw}
\left\{\begin{aligned}
    &\h_t+\dive(\h \vu)=0,\\
    &(\h \vu)_t+\dive (\h \vu\otimes \vu)+\g
    \h\nabla\h+\f(\h\vu)^\perp\\&\quad\quad\quad
    =\dive(2\xi(\h) D(\vu))
      +\nabla(\lambda(\h)\dive\vu),\\ 
    &\h(0)=\h_0, \quad \vu(0)=\vu_0,
    \end{aligned}\right.
\end{equation}
where $\h(t,x)$ is the height of the fluid surface,
$\vu(t,x)=(u^1(t,x),u^2(t,x))^\top$ is the horizontal velocity
field, $\vu^\perp(t,x)=(-u^2(t,x),u^1(t,x))$,
$x=(x_1,x_2)\in\Real^2$,
$D(\vu)=\frac{1}{2}(\nabla\vu+(\nabla\vu)^\top)$, $\g>0$ is the
gravitational acceleration, $\f>0$ is the Coriolis frequency,
$\xi\gs 0$ and $\lambda$ are the dynamical viscosities satisfying
$\lambda+\xi\gs 0$.

For the shallow water system \eqref{0sw}, there is a mount of work
to deal with the global well-posedness of strong solutions subject
to some small initial perturbation of a constant state or the global
existence of weak solutions for large initial data.
When the viscosities satisfy $\xi(\h)=\h$ and $\lambda=0$, and the
effect of the Coriolis force and/or  third-order surface tension
term is omitted ($f=0$, $\beta=0$), the local existence and uniqueness
of classical solutions to the Cauchy-Dirichlet problem for the
shallow water equations  with initial data in $C^{2+\alpha}$ was studied in
\cite{Bui81} using Lagrangian coordinates and H\"older
space estimates. Kloeden and Sundbye \cite{klo85,Sun96} proved the global existence
and uniqueness of classical solutions to the Cauchy-Dirichlet
problem using Sobolev space estimates by following the energy method
of Matsumura-Nishida \cite{MatN80}. Sundbye \cite{Sun98} proved also
the existence and uniqueness of classical solutions to the Cauchy
problem using the method of \cite{MatN80}. Wang-Xu, in
\cite{WanX05}, obtained local solutions for any initial data and
global solutions for small initial data $\h_0-\hb,\, \vu_0\in
H^{2+s}(\Real^2)$ with $s>0$. The result was improved by
Chen-Miao-Zhang and Haspot to get global existence in time for small
initial data $\h_0-\hb\in\be[0]\cap\be[1]$ and $\vu_0\in\be[0]$ as a
special case in \cite{CM08,Has08}.
Cheng-Tadmor discussed the long time existence of approximate
periodic solutions for the rapidly rotating shallow water for
initial data $(\h_0,\vu_0)\in H^m(\mathbb{T}^2)$ with $m>5$ where
the viscous terms are absent (i.e. $\xi=\lambda=0$
) in \cite{ChTa07}.
The global existence of weak solutions for arbitrarily large initial
data is established in one dimension~\cite{LiLiXin2008}, where the
vanishing of vacuum states in finite time is shown, and in
multi-dimensional bounded domain with spherical
symmetry~\cite{GuoJiuXin2007} with the help of the Bresch-Desjardins
entropy~\cite{BreschDesjardins03b} and the $L^1$-stability
compactness argument~\cite{MelletVasseur05}. The global existence of
weak solutions for arbitrarily large initial data is shown by
Bresch-Desjardins~\cite{BreschDesjardins03} where additional drag
friction and capillary terms are involved to construct an global
approximate solutions. An related systems with a third-order term
stemming from the capillary tensor also have been considered by
Danchin-Desjardins for a compressible fluid model of Korteweg type
\cite{DanD01} with constant viscosity coefficients, and the global
existence of strong solution is shown.
\par

In the present paper, we consider the global existence of the Cauchy
problem for the 2D viscous shallow water equations
\begin{equation}\label{sw}
\left\{\begin{aligned}
    &\h_t+\dive(\h \vu)=0,\\
    &(\h \vu)_t+\dive (\h \vu\otimes \vu)+\g
    \h\nabla\h+\f(\h\vu)^\perp\\
    &\quad\quad\quad=2\mu\dive(\h D(\vu))
      +2\mu\nabla(\h\dive\vu)+\beta\h\nabla\Delta\h,\\
    &\h(0)=\h_0, \quad \vu(0)=\vu_0,
    \end{aligned}\right.
\end{equation}
which corresponds to \eqref{0sw} for the case
$2\xi(\h)=\lambda(\h)=2\mu\h$ with $\mu>0$ a constant,  and is derived
recently in \cite{Mar07} with a third-order surface tension term
involved by considering second order approximation and parabolic
correction where $\beta>0$ is the capillary coefficient. Although there are many mathematical
results about the shallow water equations \eqref{0sw}, there is no
analysis about Eq.~\eqref{sw}. It also should be mentioned that
the global existence of weak solutions does not apply here since the
Bresch-Desjardins entropy~\cite{BreschDesjardins03b} is not
satisfied for Eq.~\eqref{sw}. In addition, the classical theory does not
cover the case with coriolis force and capillarity term involved.

We investigate the global existence of strong solution in some Besov
space. Although we also make use of the
Hodge's decomposition to separate the velocity field into a
compressible part and an incompressible part, unlike \cite{Dan00,Dan01} we finally obtain a
coupled system due to the rotating effect of the Coriolis force. In
fact, it can not be decoupled into a system involving only the
compressible part and a heat equation containing only the
incompressible part because of the appearance of the Coriolis
frequency, which leads to a strong coupling between the gradient
vector field part and divergence free part of the fluid velocity in
terms of the Hodge's  decomposition. Thus, we have to investigate
the whole system of the height, the compressible velocity field part
and the incompressible velocity field part. With the help of the
Littlewood-Paley analysis and hybrid Besov spaces, we obtain the a
priori estimates in Chemin-Lerner type time-spatial spaces which are
necessary in order to use the interpolation theory of time-spatial
spaces involving hybrid Besov spaces. Then we use a classical
Friedrichs' regularization method to construct approximate solutions
and prove the existence of a solution by compactness arguments. For
the uniqueness of solutions, due to the contribution of the
third-order surface tension term, we can prove it in a larger space
than that for the existence and we do not need more regularity on the spaces.

For the convenience of the statement of main results, we note that $\hbe[s_1]{s_2}$ is a
hybrid Besov space defined in the next section,  the space $E^s$ is
defined by
\begin{align*}
    E^s=&\Big\{(\h,\vu)\in
    \te[\C]\left([0,\infty);\hbe{s}\right)\cap
    L^1\left(0,\infty;\hbe[s+3]{s+2}\right)\\
    &\qquad\qquad\times\left(\te[\C]\left([0,\infty);\be[s-1]\right)
    \cap L^1\left(0,\infty;\be[s+1]\right)
    \right)^2\Big\},
\end{align*}
and $\te[\C]([0,\infty);\hbe[s_1]{s_2})$ is the subset of functions of the Chemin-Lerner type space $\te[L]_T^\infty(\hbe[s_1]{s_2})$ defined in the next section which are continuous on $[0,\infty)$ with values in $\hbe[s_1]{s_2}$.

 For the initial data $\h_0$, we suppose that it is a small
perturbation of some positive constant $\hb$. The main theorem of
this paper reads as follows.

\begin{theorem}\label{thm.main}
Let $\eps\in(0,1)$,  $\h_0-\hb\in \hbe[0]{1+\eps}$ and
$\vu_0\in\hbe[0]{\eps}$. Then, there exist two positive constants
$\alpha$ small enough and $M$ such that if
    \begin{align*}
        \norm{\h_0-\hb}_{\hbe[0]{1+\eps}}+\norm{\vu_0}_{\hbe[0]{\eps}}
        \ls\alpha,
    \end{align*}
then \eqref{sw} yields a unique global solution $(\h,\vu)$ in
$(\hb,\mathbf{0})+(E^1\cap E^{1+\eps})$ which satisfies:
\begin{align*}
    \norm{(\h-\hb,\vu)}_{E^1\cap E^{1+\eps}}
    \ls
      M(\norm{\h_0-\hb}_{\hbe[0]{1+\eps}}
    +\norm{\vu_0}_{\hbe[0]{\eps}}),
\end{align*}
where $M$ is independent of the initial data.
\end{theorem}

The paper is organized as follows. We recall some Littlewood-Paley
theories for homogeneous Besov spaces and give the definitions and
some properties of hybrid Besov spaces and Chemin-Lerner type spaces
in the second section. In Section 3, we are dedicated into proving of the a
priori estimates. In Section 4, we prove the global existence and
uniqueness of solution for small initial data by using a classical
Friedrichs' regularization method and compactness arguments.

\section{Littlewood-Paley theory and Besov spaces}
Let $\psi : \Real^2 \to [0,1]$ be a radial smooth cut-off function
valued in $[0,1]$ such that
\begin{align*}
    \psi(\xi)=\left\{
    \begin{array}{ll}
    1, &\abs{\xi}\ls 3/4,\\
    \text{smooth}, &3/4<\abs{\xi}<4/3,\\
    0, &\abs{\xi}\gs 4/3.
    \end{array}
    \right.
\end{align*}
Let $\varphi(\xi)$ be the function
\begin{align*}
    \varphi(\xi):=\psi(\xi/2)-\psi(\xi).
\end{align*}
Thus, $\psi$ is supported in the ball $\set{\xi\in\Real^2:
\abs{\xi}\ls 4/3}$, and $\varphi$ is also a smooth cut-off function
valued in $[0,1]$ and supported in the annulus $\{\xi:
3/4\ls\abs{\xi}\ls 8/3\}$.
 By construction, we have
\begin{align*}
    \sum_{k\in\Z}\varphi(2^{-k}\xi)=1, \quad \forall
    \xi\neq 0.
\end{align*}
One can define the dyadic blocks as follows. For $k\in\Z$, let
\begin{align*}
\dk f:=\F^{-1}\varphi(2^{-k}\xi)\F f.
\end{align*}
The formal decomposition
\begin{align}\label{lpd}
    f=\sum_{k\in\Z}\dk f
\end{align}
is called homogeneous Littlewood-Paley decomposition. Actually, this
decomposition works for just about any locally integrable function
which yields some decay at infinity, and one usually has all the
convergence properties of the summation that one needs. Thus, the
r.h.s. of \eqref{lpd} does not necessarily converge in
$\Sz'(\Real^2)$. Even if it does, the equality is not always true in
$\Sz'(\Real^2)$. For instance, if $f\equiv 1$, then all the
projections $\dk f$ vanish. Nevertheless, \eqref{lpd} is true modulo
polynomials, in other words (cf.\cite{Dan05,Pee76}), if
$f\in\Sz'(\Real^2)$, then $\sum_{k\in\Z}\dk f$ converges modulo
$\mathscr{P}[\Real^2]$ and \eqref{lpd} holds in
$\Sz'(\Real^2)/\mathscr{P}[\Real^2]$.

\begin{definition}
Let $s\in\Real$, $1\ls p,\,q\ls \infty$. For $f\in\Sz'(\Real^2)$, we
write
\begin{align*}
    \norm{f}_{\be[s]}=\sum_{k\in\Z}
    2^{ks}\norm{\dk f}_{L^2}.
\end{align*}
\end{definition}

A difficulty comes from the choice of homogeneous spaces at this point.
Indeed, $\norm{\cdot}_{\be[s]}$ cannot be a norm on
$\{f\in\Sz'(\Real^2): \norm{f}_{\be[s]}<\infty\}$ because
$\norm{f}_{\be[s]}=0$ means that $f$ is a polynomial. This
enforces us to adopt the following definition for homogeneous Besov
spaces (cf. \cite{Dan01}).

\begin{definition}
Let $s\in\Real$ and $m=-[2-s]$. If $m<0$, then we define
$\be[s](\Real^2)$ as
\begin{align*}
    \be[s]=\Big\{f\in\Sz'(\Real^2):
    \norm{f}_{\be[s]}<\infty \text{ and } f=\sum_{k\in\Z}\dk f
    \text{ in } \Sz'(\Real^2)\Big\}.
\end{align*}
If $m\gs 0$, we denote by $\mathscr{P}_m$ the set of two variables
polynomials of degree less than or equal to $m$ and define
\begin{align*}
    \be[s]=\Big\{f\in\Sz'(\Real^2)/\mathscr{P}_m:
    \norm{f}_{\be[s]}<\infty \text{ and }f=\sum_{k\in\Z}\dk f
    \text{ in } \Sz'(\Real^2)/\mathscr{P}_m\Big\}.
\end{align*}
\end{definition}

For the composition of functions, we have the following estimates.

\begin{lemma}[\mbox{\cite[Lemma 2.7]{Dan01}}]\label{lem.comp}
Let $s>0$ and $u\in \be[s]\cap L^\infty$.

\mbox{\rm i) } Let $F\in W_{loc}^{[s]+2,\infty}(\Real^2)$ such that
$F(0)=0$. Then $F(u)\in\be[s]$. Moreover, there exists a function of
one variable $C_0$ depending only on $s$ and $F$, and such that
\begin{align*}
    \norm{F(u)}_{\be[s]}\ls
    C_0(\norm{u}_{L^\infty})\norm{u}_{\be[s]}.
\end{align*}

\mbox{\rm ii)} If $u,\, v\in\be[1]$, $(v-u)\in \be[s]$ for a
$s\in(-1,1]$ and $G\in W_{loc}^{4,\infty}(\Real^2)$ satisfies
$G'(0)=0$, then $G(v)-G(u)\in \be[s]$ and there exists a function of
two variables $C$ depending only on $s$ and $G$, and such that
\begin{align*}
    \norm{G(v)-G(u)}_{\be[s]}\ls C(\norm{u}_{L^\infty},
    \norm{v}_{L^\infty})\left(\norm{u}_{\be[1]}+\norm{v}_{\be[1]}\right)
    \norm{v-u}_{\be[s]}.
\end{align*}
\end{lemma}

We also need hybrid Besov spaces for which regularity assumptions are different in low frequencies and high frequencies \cite{Dan01}. We are going to recall the definition of these new spaces and some of their main properties.

\begin{definition}
Let $s,\,t\in\Real$. We define
\begin{align*}
    \norm{f}_{\hbe[s]{t}}=\sum_{k\ls 0}2^{ks}\norm{\dk f}_{L^2}
    +\sum_{k>0}2^{kt}\norm{\dk f}_{L^2}.
\end{align*}
Let $m=-[2-s]$, we then define
\begin{align*}
    \hbe[s]{t}(\Real^2)=&\set{f\in\Sz'(\Real^2): \norm{f}_{\hbe[s]{t}}<\infty}, \quad \text{if } m<0,\\
    \hbe[s]{t}(\Real^2)=&\set{f\in\Sz'(\Real^2)/\mathscr{P}_m: \norm{f}_{\hbe[s]{t}}<\infty}, \quad \text{if } m \gs 0.
\end{align*}
\end{definition}

\begin{lemma} We have the following inclusions.

(i) We have $\hbe[s]{s}=\be[s]$.

(ii) If $s\ls t$ then $\hbe[s]{t}=\be[s]\cap\be[t]$. Otherwise, $\hbe[s]{t}=\be[s]+\be[t]$.

(iii) The space $\hbe[0]{s}$ coincides with the usual inhomogeneous Besov space $B_{2,1}^s$.

(iv) If $s_1\ls s_2$ and $t_1\gs t_2$, then $\hbe[s_1]{t_1}\hookrightarrow \hbe[s_2]{t_2}$.
\end{lemma}

Let us now recall some useful estimates for the product in hybrid Besov spaces.

\begin{lemma}[\mbox{\cite[Proposition 2.10]{Dan01}}]\label{lem.fgbesov}
Let $s_1,\, s_2>0$ and $f,\,g\in L^\infty\cap \hbe[s_1]{s_2}$. Then $fg\in \hbe[s_1]{s_2}$ and
\begin{align*}
    \norm{fg}_{\hbe[s_1]{s_2}}\lesssim
    \norm{f}_{L^\infty}\norm{g}_{\hbe[s_1]{s_2}}
    +\norm{f}_{\hbe[s_1]{s_2}}\norm{g}_{L^\infty}.
\end{align*}
Let $s_1, s_2, t_1, t_2\ls 1$ such that $\min(s_1+s_2, t_1+t_2)>0$, $f\in \hbe[s_1]{t_1}$ and $g\in
\hbe[s_2]{t_2}$. Then $fg\in \hbe[s_1+s_2-1]{t_1+t_2-1}$ and
\begin{align*}
    \norm{fg}_{\hbe[s_1+s_2-1]{t_1+t_2-1}}\lesssim
    \norm{f}_{\hbe[s_1]{t_1}}\norm{g}_{\hbe[s_2]{t_2}}.
\end{align*}
\end{lemma}

In the context of this paper, we also need to use the interpolation spaces of hybrid Besov spaces together with a time space such as $L^p(0,T;\hbe[s]{t})$. Thus, we have to introduce the Chemin-Lerner type space which is a refinement of the space $L^p(0,T;\hbe[s]{t})$.

\begin{definition}
Let $p\in [1,\infty]$, $T\in(0,\infty]$ and $s_1,\, s_2\in\Real$. Then we define
\begin{align*}
    \norm{f}_{\te[L]_T^p(\hbe[s]{t})}=\sum_{k\ls 0}2^{ks}\norm{\dk f}_{L^p(0,T; L^2)}+\sum_{k> 0}2^{kt}\norm{\dk f}_{L^p(0,T; L^2)}.
\end{align*}
\end{definition}

Noting that Minkowski's inequality yields $\norm{f}_{L_T^p(\hbe[s]{t})}\ls \norm{f}_{\te[L]_T^p(\hbe[s]{t})}$, we define spaces $\te[L]_T^p(\hbe[s]{t})$ as follows
\begin{align*}
    \te[L]_T^p(\hbe[s]{t})=\set{f\in L_T^p(\hbe[s]{t}): \norm{f}_{\te[L]_T^p(\hbe[s]{t})}<\infty}.
\end{align*}
If $T=\infty$, then we omit the subscript $T$ from the notation $\te[L]_T^p(\hbe[s]{t})$, that is, $\te[L]^p(\hbe[s]{t})$ for simplicity. We will denote by $\te[\C]([0,T];\hbe[s]{t})$ the subset of functions of $\te[L]_T^\infty(\hbe[s]{t})$ which are continuous on $[0,T]$ with values in $\hbe[s]{t}$.

Let us observe that $L_T^1(\hbe[s]{t})=\te[L]_T^1(\hbe[s]{t})$, but the embedding $\te[L]_T^p(\hbe[s]{t})\subset L_T^p(\hbe[s]{t})$ is strict if $p>1$.

We will use the following interpolation property which can be verified easily (cf. \cite{Ber76}).

\begin{lemma}\label{lem.inter}
Let $s,t,s_1,t_1, s_2,t_2\in\Real$ and $p,p_1,p_2\in[1,\infty]$. We have
\begin{align*}
    \norm{f}_{\te[L]_T^p(\hbe[s]{t})}\ls \norm{f}_{\te[L]_T^{p_1}(\hbe[s_1]{t_1})}^\theta \norm{f}_{\te[L]_T^{p_2}(\hbe[s_2]{t_2})}^{1-\theta},
\end{align*}
where $\frac{1}{p}=\frac{\theta}{p_1}+\frac{1-\theta}{p_2}$, $s=\theta s_1+(1-\theta)s_2$ and $t=\theta t_1+(1-\theta)t_2$.
\end{lemma}

Now, we define the following work space.

\begin{definition}\label{est.def}
For $T>0$ and $s\in\Real$, we denote
\begin{align*}
    E_T^s=&\Big\{(\h,\vu)\in
    \te[\C]\left([0,T];\hbe{s}\right)\cap
    L^1\left(0,T;\hbe[s+3]{s+2}\right)\\
    &\qquad\qquad\times\left(\te[\C]\left([0,T];\be[s-1]\right)
    \cap L^1\left(0,T;\be[s+1]\right)
    \right)^2\Big\}
\end{align*}
and
\begin{align*}
\norm{(\h,\vu)}_{E_T^s}=&\norm{\h}_{\te[L]_T^\infty(\hbe{s})
}+\norm{\vu}_{\te[L]_T^\infty(\be[s-1])}\\
&+\norm{\h}_{L_T^1(\hbe[s+3]{s+2})}
+\norm{\vu}_{L_T^1(\be[s+1])}.
\end{align*}

We use the notation $E^s$ if $T=+\infty$, changing $[0,T]$ into
$[0,+\infty)$ in the definition above.
\end{definition}

\section{A priori estimates}

Noticing that $ \dive
D(\vu)=\frac{1}{2}\nabla\dive\vu+\frac{1}{2}\Delta\vu$ and
substituting $\h$ by $\h+\hb$ in \eqref{sw}, we have
\begin{align}\label{swm}
\left\{\begin{aligned}
    &\h_t+\vu\cdot\nabla\h+\hb\dive\vu=-\h\dive\vu,\\
    &\vu_t+\vu\cdot\nabla\vu-\mu\Delta\vu-3\mu\nabla\dive\vu
    +\f\vu^\perp+\g\nabla\h-\beta\nabla\Delta\h\\
    &\qquad\qquad\qquad
    =2\mu\frac{\nabla\h D(\vu)+\nabla\h\dive\vu}{\h+\hb},\\
    &\h(0)=\h_0-\hb, \quad \vu(0)=\vu_0.
    \end{aligned}\right.
\end{align}

For all $s\in\Real$, we denote
$\lam^sf=\F^{-1}(\abs{\xi}^s\hat{f})$. Let $c=\lam^{-1}\dive\vu$ and
$\inc=\lam^{-1}\divp\vu$ where $\divp \vu=\nabla^\perp\cdot\vu$ and
$\nabla^\perp=(-\partial_2,\partial_1)$. Then, it is easy to check
that
\begin{align*}
    \vu=-\lam^{-1}\nabla c-\lam^{-1}\nabla^\perp \inc.
\end{align*}

Now, we can rewrite the system \eqref{swm} in terms of these
notations as the following:
\begin{align}\label{swm1}
\left\{\begin{aligned}
    &\h_t+\vu\cdot\nabla\h+\hb\lam c=F,\\
    &c_t+\vu\cdot\nabla c-4\mu\Delta c-\f\inc-\g\lam\h
    -\beta\lam^3\h=G,\\
    &\inc_t-\mu\Delta\inc+\f c=\lam^{-1}\divp H,\\
    &\vu=-\lam^{-1}\nabla c-\lam^{-1}\nabla^\perp\inc,\\
    &\h(0)=\h_0-\hb, \quad \vu(0)=\vu_0,
    \end{aligned}\right.
\end{align}
where
\begin{align*}
    &F=-\h\dive\vu,\\
    &G=\vu\cdot\nabla c+\lam^{-1}\dive H,\\
    &H=-\vu\cdot\nabla\vu+2\mu\frac{\nabla\h D(\vu)
    +\nabla\h\dive\vu}{\h+\hb}.
\end{align*}

For these equations, we study the following system:
\begin{align}\label{swm2}
  \left\{\begin{aligned}
    &\h_t+\vv\cdot\nabla\h+\hb\lam c=F,\\
    &c_t+\vv\cdot\nabla c-4\mu\Delta c-\f\inc-\g\lam\h
    -\beta\lam^3\h=G,\\
    &\inc_t-\mu\Delta\inc+\f c=P,
  \end{aligned}\right.
\end{align}
where $\vv$ is a vector function and we will precise its regularity
in the following proposition.

\begin{proposition}\label{prop.conven}
   Let $(\h,c,\inc)$ be a solution of \eqref{swm2} on $[0,T)$, $T>0$, $0<s\ls
2$ and $V(t)=\int_0^t\norm{\vv(\tau)}_{\be[2]}d\tau$. The following
estimate holds on $[0,T)$:
\begin{align*}
    &\norm{\h}_{\te[L]_T^\infty(\hbe{s})}
    +\norm{c}_{\te[L]_T^\infty(\be[s-1])}
    +\norm{\inc}_{\te[L]_T^\infty(\be[s-1])}\\
    &\qquad+\int_0^t\Big(\norm{\h(\tau)}_{\hbe[s+3]{s+2}}
    +\norm{c(\tau)}_{\be[s+1]}+\norm{\inc(\tau)}_{\be[s+1]}\Big)d\tau\\
    \ls C &e^{CV(t)}\Big(\norm{\h(0)}_{\hbe{s}}
    +\norm{c(0)}_{\be[s-1]}+\norm{\inc(0)}_{\be[s-1]}\\
    &\qquad+\int_0^te^{-CV(\tau)}\left(\norm{F(\tau)}_{\hbe{s}}
    +\norm{G(\tau)}_{\be[s-1]}+\norm{P(\tau)}_{\be[s-1]}\right)d\tau\Big),
\end{align*}
where $C$ depends only on $s$, $\hb$ and coefficients $\mu$, $\f$, $\g$ and $\beta$.
\end{proposition}

\begin{proof}
Let $(\h, c, \inc)$ be a solution of \eqref{swm2} and we set
\begin{align*}
    (\te[\h],\te[c],\te[\inc], \te[F], \te[G], \te[P])=e^{-KV(t)}(\h, c, \inc, F, G, P).
\end{align*}
Thus, \eqref{swm2} can be transformed into
\begin{align}\label{swm3}
  \left\{\begin{aligned}
    &\te[\h]_t+\vv\cdot\nabla\te[\h]+\hb\lam \te[c]=\te[F]-KV'(t)\te[\h],\\
    &\te[c]_t+\vv\cdot\nabla \te[c]-4\mu\Delta \te[c]-\f\te[\inc]-\g\lam\te[\h]
    -\beta\lam^3\te[\h]=\te[G]-KV'(t)\te[c],\\
    &\te[\inc]_t-\mu\Delta\te[\inc]+\f\te[c]=\te[P]-KV'(t)\te[\inc].
  \end{aligned}\right.
\end{align}

Applying the operator $\dk$ to the system \eqref{swm3},  we obtain
the following system by noting $(\te[\h]_k, \te[c]_k, \te[\inc]_k,
\te[F]_k, \te[G]_k, \te[P]_k)=(\dk\te[\h], \dk\te[c], \dk\te[\inc],
\dk\te[F],\dk\te[G], \dk\te[P])$:
\begin{align}\label{swm3k}
  \left\{\begin{aligned}
    &\partial_t\te[\h]_k+\dk(\vv\cdot\nabla\te[\h])+\hb\lam \te[c]_k
    =\te[F]_k-KV'(t)\te[\h]_k,\\
    &\partial_t\te[c]_k+\dk(\vv\cdot\nabla \te[c])-4\mu\Delta
    \te[c]_k-\f\te[\inc]_k-\g\lam\te[\h]_k
    -\beta\lam^3\te[\h]_k=\te[G]_k-KV'(t)\te[c]_k,\\
    &\partial_t\te[\inc]_k-\mu\Delta\te[\inc]_k+\f\te[c]_k=\te[P]_k
    -KV'(t)\te[\inc]_k.
  \end{aligned}\right.
\end{align}

To begin with, we consider the case where $\vv=\mathbf{0}$, $K=0$
and $F=G=P=0$ which implies that \eqref{swm3k} takes the form
\begin{align}\label{swm4}
  \left\{\begin{aligned}
    &\partial_t\te[\h]_k+\hb\lam \te[c]_k
    =0,\\
    &\partial_t\te[c]_k-4\mu\Delta
    \te[c]_k-\f\te[\inc]_k-\g\lam\te[\h]_k
    -\beta\lam^3\te[\h]_k=0,\\
    &\partial_t\te[\inc]_k-\mu\Delta\te[\inc]_k+\f\te[c]_k=0.
  \end{aligned}\right.
\end{align}

\subsection{The case of high frequencies}
Taking the $L^2$ scalar product of the first equation of
\eqref{swm4} with $\te[\h]_k$, of the second equation with
$\te[c]_k$, and the third one with $\te[\inc]_k$,  we get the
following two identities:
\begin{align}\label{swm41}
  \left\{\begin{aligned}
    &\frac{1}{2}\frac{d}{dt}\norm{\te[\h]_k}_{L^2}^2+\hb(\lam\te[c]_k,\te[\h]_k)=0,\\
    &\frac{1}{2}\frac{d}{dt}\norm{\te[c]_k}_{L^2}^2+4\mu\norm{\lam\te[c]_k}_{L^2}^2
    -\f(\te[\inc]_k,\te[c]_k)-\g(\lam\te[\h]_k,\te[c]_k)\\
    &\qquad\qquad\qquad\qquad\qquad-\beta(\lam^2\te[\h]_k,\lam\te[c]_k)=0,\\
    &\frac{1}{2}\frac{d}{dt}\norm{\te[\inc]_k}_{L^2}^2
    +\mu\norm{\lam\te[\inc]_k}_{L^2}^2+\f(\te[c]_k,\te[\inc]_k)=0.
  \end{aligned}\right.
\end{align}
Now we want to get an equality involving $\lam\te[\h]_k$. To achieve it, we take $L^2$ scalar product of the first equation of
\eqref{swm4} with $\lam^2\te[\h]_k$ and $\lam\te[c]_k$ respectively, then take the $L^2$ scalar product of the second equation with
$\lam\te[\h]_k$ and sum with both last two equalities. This yields
\begin{align}\label{swm42}
   \left\{\begin{aligned}
   &\frac{1}{2}\frac{d}{dt}\norm{\lam\te[\h]_k}_{L^2}^2
   +\hb(\lam^2\te[c]_k,\lam\te[\h]_k)=0,\\
   &\frac{d}{dt}(\lam\te[\h]_k,\te[c]_k)+\hb\norm{\lam\te[c]_k}_{L^2}^2
    -\f(\te[\inc]_k,\lam\te[\h]_k)-\g\norm{\lam \te[\h]_k}_{L^2}^2\\
    &\qquad\qquad\qquad-\beta\norm{\lam^2 \te[\h]_k}_{L^2}^2
    +4\mu(\lam\te[c]_k,\lam^2\te[\h]_k)=0.
    \end{aligned}\right.
\end{align}
Let $K_1>0$ be a constant to be chosen later and denote for $k>0$
\begin{align*}
    \alpha_k^2=\frac{\g}{\hb}\norm{\te[\h]_k}_{L^2}^2
    +\frac{\beta}{\hb}\norm{\lam\te[\h]_k}_{L^2}^2
    +\norm{\te[c]_k}_{L^2}^2+\norm{\te[\inc]_k}_{L^2}^2-2K_1(\lam\te[\h]_k,\te[c]_k).
\end{align*}
By a linear combination of \eqref{swm41} and \eqref{swm42}, we can
get
\begin{align}\label{swm43}
    \begin{aligned}
    \frac{1}{2}\frac{d}{dt}\alpha_k^2
    &+(4\mu-\hb K_1)\norm{\lam\te[c]_k}_{L^2}^2
    +\g K_1\norm{\lam \te[\h]_k}_{L^2}^2
    +\beta K_1\norm{\lam^2 \te[\h]_k}_{L^2}^2\\
    &+\mu \norm{\lam\te[d]_k}_{L^2}^2-4\mu
    K_1(\lam\te[c]_k,\lam^2\te[\h]_k)+\f K_1(\te[\inc]_k,\lam\te[\h]_k)
    =0.
    \end{aligned}
\end{align}

 Using Schwartz' inequality, Young's inequality and Bernstein's inequality
\begin{align*}
 \|\te[\inc]_k\|_{L^2} \ls \frac{4}{3}2^{-k}\|\lam\te[\inc]_k\|_{L^2},
\end{align*}
 we find,
for any positive numbers $M_1$, $M_2$, $M_3$, that
\begin{align*}
    \abs{(\lam\te[c]_k,\lam^2\te[\h]_k)}\ls& \frac{M_1}{2}
    \norm{\lam\te[c]_k}_{L^2}^2
    +\frac{1}{2M_1}\norm{\lam^2\te[\h]_k}_{L^2}^2,\\
    \abs{(\te[\inc]_k,\lam\te[\h]_k)}\ls& \frac{8M_2}{9}
    \norm{\lam\te[\inc]_k}_{L^2}^2
    +\frac{1}{2M_2}\norm{\lam\te[\h]_k}_{L^2}^2,\\
    \abs{(\lam\te[\h]_k,\te[c]_k)}
    \ls&\frac{M_3}{2}\norm{\lam\te[\h]_k}_{L^2}^2
    +\frac{1}{2M_3}\norm{\te[c]_k}_{L^2}^2.
\end{align*}
Thus, we need to determine the values of $K_1$, $M_1$, $M_2$
and $M_3$ such that
\begin{align*}
    4\mu-\hb K_1-4\mu K_1\frac{M_1}{2}>0, \quad
    \beta-\frac{4\mu}{2M_1}>0,\quad \g-\frac{\f}{2M_2}>0,\\
    \mu-\frac{8\f K_1M_2}{9}>0, \quad \frac{\beta}{\hb}-K_1 M_3>0, \quad
    1-\frac{K_1}{M_3}>0.
\end{align*}
One can verify that the above inequalities will hold if one has
\begin{align*}
&0<K_1<\min\left(\frac{4\mu\beta}{\hb\beta+5\mu^2},
\frac{2}{3}\sqrt{\frac{\beta}{\hb}},\frac{9\mu\g}{4\f^2}\right),\\
&M_1=\frac{5\mu}{2\beta},\;
 M_2=\frac{\f}{4\g}+\frac{9\mu}{16\f K_1},\;
M_3=\frac{2}{3}\sqrt{\frac{\beta}{\hb}}.
\end{align*}
Hence, we obtain
\begin{align}\label{swm44z}
    c_1\alpha_k^2\ls \norm{\te[\h]_k}_{L^2}^2
    +\norm{\lam\te[\h]_k}_{L^2}^2
    +\norm{\te[c]_k}_{L^2}^2+\norm{\te[\inc]_k}_{L^2}^2
    \ls c_2\alpha_k^2.
\end{align}

Therefore, there exists a constant $\dot{c}>0$ such that
\begin{align*}
    \frac{1}{2}\frac{d}{dt}\alpha_k^2+\dot{c}2^{2k}\alpha_k^2\ls
    0.
\end{align*}

In the general case where $F$, $G$, $P$, $K$ and $\vv$ are not zero,
we have, with the help of Lemma 6.2 in \cite{Dan01}, that
\begin{align*}
    &\frac{1}{2}\frac{d}{dt}\alpha_k^2+\left(\dot{c}2^{2k}+KV'\right)\alpha_k^2\\
    \ls&
    \frac{\g}{\hb}(\te[F]_k,\te[\h]_k)+(\te[G]_k,\te[c]_k)
    +(\te[P]_k,\te[c]_k)\\
    &+\frac{\beta}{\hb}
    (\lam\te[F]_k,\lam\te[\h]_k)
    -K_1(\lam\te[F]_k,\te[c]_k)
    -K_1(\te[G]_k,\lam\te[\h]_k)\\
    &-\frac{\g}{\hb}(\dk(\vv\cdot\nabla\te[\h]),\te[\h]_k)
    -(\dk(\vv\cdot\nabla\te[c]),\te[c]_k)
    -\frac{\beta}{\hb}(\lam\dk(\vv\cdot\nabla\te[\h]),\lam\te[\h]_k)\\
    &+K_1(\lam\dk(\vv\cdot\nabla\te[\h]),\te[c]_k)
    +K_1(\dk(\vv\cdot\nabla\te[c]),\lam\te[\h]_k) \\
    \lesssim& \alpha_k\Big(\norm{\te[F]_k}_{L^2}+\norm{\te[G]_k}_{L^2}
    +\norm{\te[P]_k}_{L^2}+\norm{\lam\te[F]_k}_{L^2}\\
    &+\gamma_k2^{-k(s-1)}
    \norm{\vv}_{\be[2]}\norm{\te[\h]}_{\be[s-1]}
    +\gamma_k2^{-k(s-1)}
    \norm{\vv}_{\be[2]}\norm{\te[c]}_{\be[s-1]}\\
    &+\gamma_k2^{-k(s-1)}
    \norm{\vv}_{\be[2]}\norm{\te[\h]}_{\be[s]}+\gamma_k\norm{\vv}_{\be[2]}(2^{-k(s-1)}
    \norm{\te[c]}_{\be[s-1]}\\
    &+2^{-k(s-1)}
    \norm{\te[\h]}_{\be[s]})
    \Big),
\end{align*}
where $\sum_k \gamma_k\ls 1$ and $s\in (0, 2]$.

\subsection{The case of low frequencies}
We replace the second equation of \eqref{swm42} by the following equation
\begin{align}\label{swm45}
\begin{aligned}
    &\frac{d}{dt}(\lam^3\te[\h]_k,\te[c]_k)+\hb\norm{\lam^2\te[c]_k}_{L^2}^2
    -\f(\lam\te[\inc]_k,\lam^2\te[\h]_k)-\g\norm{\lam^2 \te[\h]_k}_{L^2}^2\\
    &\qquad\qquad\qquad-\beta\norm{\lam^3 \te[\h]_k}_{L^2}^2
    +4\mu(\lam^2\te[c]_k,\lam^3\te[\h]_k)=0.
\end{aligned}
\end{align}
Let $K_2>0$ be a constant to be chosen later and denote for $k\ls 0$
\begin{align*}
    \alpha_k^2=\frac{\g}{\hb}\norm{\te[\h]_k}_{L^2}^2
    +\frac{\beta}{\hb}\norm{\lam\te[\h]_k}_{L^2}^2
    +\norm{\te[c]_k}_{L^2}^2+\norm{\te[\inc]_k}_{L^2}^2-2K_2(\lam^3\te[\h]_k,\te[c]_k).
\end{align*}
A linear combination of \eqref{swm41}, the first equation of \eqref{swm42} and \eqref{swm45} yields
\begin{align}\label{swm46}
    \begin{aligned}
    \frac{1}{2}\frac{d}{dt}\alpha_k^2
    &+4\mu\norm{\lam\te[c]_k}_{L^2}^2-\hb K_2\norm{\lam^2\te[c]_k}_{L^2}^2
    +\g K_2\norm{\lam^2 \te[\h]_k}_{L^2}^2\\
    &   +\beta K_2\norm{\lam^3 \te[\h]_k}_{L^2}^2 +\mu \norm{\lam\te[d]_k}_{L^2}^2-4\mu
    K_2(\lam^2\te[c]_k,\lam^3\te[\h]_k)\\
    &+\f K_2(\lam\te[\inc]_k,\lam^2\te[\h]_k)
    =0.
    \end{aligned}
\end{align}

Using Schwartz' inequality, Young's inequality and Bernstein's inequality
\begin{align*}
\|\lam\te[\h]_k\|_{L^2}\ls \frac{8}{3}2^{k}\|\te[\h]_k\|_{L^2},
\end{align*}
we find,
for any positive numbers $M_4$, $M_5$, $M_6$, that
\begin{align*}
    \abs{(\lam^2\te[c]_k,\lam^3\te[\h]_k)}\ls& \frac{8^2M_4}{2\cdot 3^2}
    \norm{\lam\te[c]_k}_{L^2}^2
    +\frac{1}{2M_4}\norm{\lam^3\te[\h]_k}_{L^2}^2,\\
    \abs{(\lam\te[\inc]_k,\lam^2\te[\h]_k)}\ls& \frac{M_5}{2}
    \norm{\lam\te[\inc]_k}_{L^2}^2
    +\frac{1}{2M_5}\norm{\lam^2\te[\h]_k}_{L^2}^2,\\
    \abs{(\lam^3\te[\h]_k,\te[c]_k)}
    \ls&\frac{8^4M_6}{2\cdot 3^4}\norm{\lam\te[\h]_k}_{L^2}^2
    +\frac{1}{2M_6}\norm{\te[c]_k}_{L^2}^2.
\end{align*}
Thus, we need to determine the values of $K_2$, $M_4$, $M_5$
and $M_6$ such that
\begin{align*}
    4\mu-\frac{8^2}{3^2}\hb K_2-4\mu K_2\frac{8^2M_4}{2\cdot 3^2}>0, \quad
    \beta-\frac{4\mu}{2M_4}>0,\quad \g-\frac{\f}{2M_5}>0,\\
    \mu-\frac{\f K_2M_5}{2}>0, \quad \frac{\beta}{\hb}-\frac{8^4}{2\cdot 3^4}K_2 M_6>0, \quad
    1-\frac{K_2}{M_6}>0.
\end{align*}
One can verify that the above inequalities will hold if one chooses
\begin{align*}
&0<K_2<\min\left(\frac{3^2\cdot4\mu\beta}{8^2(\hb\beta+5\mu^2)},
\frac{3^2}{8^2}\sqrt{\frac{\beta}{\hb}},\frac{4\mu\g}{\f^2}\right),\\
&M_4=\frac{5\mu}{2\beta},\;
 M_5=\frac{\f}{4\g}+\frac{\mu}{\f K_2},\;
M_6=\frac{3^2}{8^2}\sqrt{\frac{\beta}{\hb}}.
\end{align*}
Hence, we obtain
\begin{align}\label{swm44}
    c_3\alpha_k^2\ls \norm{\te[\h]_k}_{L^2}^2
    +\norm{\lam\te[\h]_k}_{L^2}^2
    +\norm{\te[c]_k}_{L^2}^2+\norm{\te[\inc]_k}_{L^2}^2
    \ls c_4\alpha_k^2.
\end{align}

Therefore, there exists a constant $\ddot{c}>0$ such that
\begin{align*}
    \frac{1}{2}\frac{d}{dt}\alpha_k^2+\ddot{c}2^{4k}\alpha_k^2\ls
    0.
\end{align*}

In the general case where $F$, $G$, $P$, $K$ and $\vv$ are not zero,
we have, with the help of Lemma 6.2 in \cite{Dan01}, that
\begin{align*}
    &\frac{1}{2}\frac{d}{dt}\alpha_k^2+\left(\ddot{c}2^{4k}+KV'\right)\alpha_k^2\\
    \ls&
    \frac{\g}{\hb}(\te[F]_k,\te[\h]_k)+(\te[G]_k,\te[c]_k)
    +(\te[P]_k,\te[c]_k)\\
    &+\frac{\beta}{\hb}
    (\lam\te[F]_k,\lam\te[\h]_k)
    -K_1(\lam^3\te[F]_k,\te[c]_k)
    -K_1(\te[G]_k,\lam^3\te[\h]_k)\\
    &-\frac{\g}{\hb}(\dk(\vv\cdot\nabla\te[\h]),\te[\h]_k)
    -(\dk(\vv\cdot\nabla\te[c]),\te[c]_k)
    -\frac{\beta}{\hb}(\lam\dk(\vv\cdot\nabla\te[\h]),\lam\te[\h]_k)\\
    &+K_1(\lam^3\dk(\vv\cdot\nabla\te[\h]),\te[c]_k)
    +K_1(\dk(\vv\cdot\nabla\te[c]),\lam^3\te[\h]_k) \\
    \lesssim& \alpha_k\Big(\norm{\te[F]_k}_{L^2}+\norm{\te[G]_k}_{L^2}
    +\norm{\te[P]_k}_{L^2}\\
    &+\gamma_k2^{-k(s-1)}
    \norm{\vv}_{\be[2]}\norm{\te[\h]}_{\be[s-1]}
    +\gamma_k2^{-k(s-1)}
    \norm{\vv}_{\be[2]}\norm{\te[c]}_{\be[s-1]}\\
    &+\gamma_k2^{-k(s-1)}
    \norm{\vv}_{\be[2]}\norm{\te[\h]}_{\be[s]}+\gamma_k\norm{\vv}_{\be[2]}(2^{-k(s-1)}
    \norm{\te[c]}_{\be[s-1]}\\
    &+2^{-k(s-1)}
    \norm{\te[\h]}_{\be[s-1]})
    \Big),
\end{align*}
where $\sum_k \gamma_k\ls 1$ and $s\in (0, 3]$.

Thus, combining two case of high and low frequencies, we obtain for any $k\in\Z$
\begin{align}\label{swm5}
\begin{aligned}
    &\frac{1}{2}\frac{d}{dt}\alpha_k^2
    +\left(\breve{c}2^{2k}\min(1,2^{2k})+KV'\right)\alpha_k^2\\
    \lesssim&\alpha_k\Big(\norm{\te[F]_k}_{L^2}+\norm{\te[G]_k}_{L^2}
    +\norm{\te[P]_k}_{L^2}
    +\norm{\lam\te[F]_k}_{L^2}\\
    &+\gamma_k2^{-k(s-1)}V'\norm{(\te[\h],\te[c])}_{\hbe{s}\times\be[s-1]}\Big),
\end{aligned}
\end{align}
where we choose $\breve{c}=\min(\dot{c},\ddot{c})$ and $V(t)=\int_0^t\norm{\vv}_{\be[2]}$.

We are now going to show that the inequality \eqref{swm5} implies a
decay for $\h$, $c$ and $\inc$.

\subsection{The damping effect for $\h$}
Dividing \eqref{swm5} by $\alpha_k$, we get
\begin{align}\label{swm6z}
\begin{aligned}
    &\frac{d}{dt}\alpha_k(t)
    +\left(\breve{c}2^{2k}\min(1,2^{2k})+KV'\right)\alpha_k\\
    \lesssim&\norm{\te[F]_k}_{L^2}+\norm{\te[G]_k}_{L^2}
    +\norm{\te[P]_k}_{L^2}
    +\norm{\lam\te[F]_k}_{L^2}
    +\gamma_k2^{-k(s-1)}V'
    \norm{(\te[\h],\te[c])}_{\hbe{s}\times\be[s-1]}.
\end{aligned}
\end{align}
Integrating over $[0,t]$, we have
\begin{align}\label{swm61}
\begin{aligned}
    &\alpha_k(t)
    +\breve{c}2^{2k}\min(1,2^{2k})\int_0^t \alpha_k(\tau)d\tau\\
    \lesssim&\alpha_k(0)+\int_0^t\Big(\norm{\te[F]_k(\tau)}_{L^2}
    +\norm{\te[G]_k(\tau)}_{L^2}
    +\norm{\te[P]_k(\tau)}_{L^2}
    +\norm{\lam\te[F]_k(\tau)}_{L^2}\Big)d\tau\\
    &+\int_0^tV'(\tau) \Big[\gamma_k(\tau)2^{-k(s-1)}
    \norm{(\te[\h],\te[c])}_{\hbe{s}\times\be[s-1]}-K\alpha_k(\tau)\Big]d\tau.
\end{aligned}
\end{align}
By the definition of $\alpha_k^2$, we have
\begin{align}\label{swm62}
\begin{aligned}
    2^{k(s-1)}\alpha_k\approx & 2^{ks}\max(1,2^{-k})\norm{\te[\h]_k}_{L^2}
    +2^{k(s-1)}\norm{\te[c]_k}_{L^2}\\
   & +2^{k(s-1)}\norm{\te[\inc]_k}_{L^2}, \quad \forall k\in\Z.
\end{aligned}
\end{align}
Thus, we have, by taking $K$ large enough, that
\begin{align*}
    \sum_{k\in\Z}\Big[\gamma_k(\tau)
    \norm{(\te[\h],\te[c])}_{\hbe{s}\times\be[s-1]}
    -K2^{k(s-1)}\alpha_k(\tau)\Big]\ls 0.
\end{align*}
Changing the functions $(\te[\h],\te[c],\te[\inc], \te[F],\te[G],
\te[P])$ into the original ones $(\h,c,\inc,F,G,P)$ and multiplying
both sides of \eqref{swm61} by $2^{k(s-1)}$. According to the last
inequality, and due to \eqref{swm61} and \eqref{swm62}, we
conclude after summation on $k$ in $\Z$, that
\begin{align}\label{swm63}
    \begin{aligned}
    &\norm{\h}_{\te[L]_T^\infty(\hbe{s})}
    +\norm{c}_{\te[L]_T^\infty(\be[s-1])}
    +\norm{\inc}_{\te[L]_T^\infty(\be[s-1])}+\breve{c}\int_0^t \norm{\h(\tau)}_{\hbe[s+3]{s+2}}d\tau\\
    &\qquad\qquad+\breve{c}\sum_{k\in\Z}\int_0^t 2^{k(s+1)}\min(1,2^{2k})\norm{c_k(\tau)}_{L^2}d\tau\\
    &\qquad\qquad+\breve{c}\sum_{k\in\Z}\int_0^t 2^{k(s+1)}\min(1,2^{2k})\norm{\inc_k(\tau)}_{L^2}d\tau\\
    \lesssim &e^{CV(t)}\norm{(\h(0),c(0),\inc(0))}_{\hbe{s}
    \times(\be[s-1])^2} \\ &\qquad\qquad+e^{CV(t)}\int_0^te^{-CV(\tau)}\norm{(F,G,P)(\tau)}_{\hbe{s}
    \times(\be[s-1])^2}d\tau.
    \end{aligned}
\end{align}

\subsection{The smoothing effects of $c$ and $\inc$}

Once the
damping effect for $h$ is established, it is easy to get the smoothing effect on
$c$ and $\inc$. Since \eqref{swm63} implies the desired estimate for high
frequencies, it suffices to prove it for low frequencies only. We
therefore suppose in this part that $k\ls 0$.

Taking the $L^2$ scalar product of the last two equations of \eqref{swm3k} with $\te[c]_k$ and $\te[\inc]_k$ respectively, we have
\begin{align}\label{swm7}
\left\{\begin{aligned}
    &\frac{1}{2}\frac{d}{dt}\norm{\te[\h]_k}_{L^2}^2
    +\hb(\te[c]_k,\lam \te[\h]_k)\\
    &\qquad=(\te[F]_k,\te[\h]_k)-KV'(t)\norm{\te[\h]_k}_{L^2}^2
    -(\dk(\vv\cdot\nabla\te[\h]),\te[\h]_k),\\
    &\frac{1}{2}\frac{d}{dt}\norm{\te[c]_k}_{L^2}^2
    +4\mu\norm{\lam\te[c]_k}_{L^2}^2
    -\f(\te[\inc]_k,\te[c]_k)-\g(\lam\te[\h]_k,\te[c]_k)
    -\beta(\lam^3\te[\h]_k,\te[c]_k)\\
    &\qquad=(\te[G]_k,\te[c]_k)-KV'(t)\norm{\te[c]_k}_{L^2}^2
    -(\dk(\vv\cdot\nabla\te[c]),\te[c]_k),\\
    &\frac{1}{2}\frac{d}{dt}\norm{\te[\inc]_k}_{L^2}^2
    +\mu\norm{\lam\te[\inc]_k}_{L^2}^2
    +\f(\te[\inc]_k,\te[c]_k)
    =(\te[P]_k,\te[\inc]_k)-KV'(t)\norm{\te[\inc]_k}_{L^2}^2.
\end{aligned}\right.
\end{align}

Define $\theta_k^2=\frac{\g}{\hb}\norm{\te[\h]_k}_{L^2}^2+\norm{\te[c]_k}_{L^2}^2+\norm{\te[\inc]_k}_{L^2}^2$. By using Lemma 6.2 in \cite{Dan01}, \eqref{swm7} yields, for a constant $c>0$, that
\begin{align*}
    \frac{1}{2}\frac{d}{dt}\theta_k^2+C2^{2k}\theta_k^2\lesssim &\theta_k(\norm{\lam^3\te[\h]_k}_{L^2}+\norm{\te[G]_k}_{L^2}
    +\norm{\te[P]_k}_{L^2})\\
    &+\theta_kV'(t)(C\gamma_k2^{-k(s-1)}
    (\norm{\te[c]}_{\be[s-1]}
    +\norm{\te[\h]}_{\be[s-1]})-K\theta_k).
\end{align*}
Dividing by $\theta_k$ and integrating over $[0,t]$, we infer
\begin{align*}
    \theta_k(t)&+C\int_0^t2^{2k}\theta_k(\tau)d\tau\ls \theta_k(0)
    +C\int_0^t[\norm{\te[G]_k(\tau)}_{L^2}
    +\norm{\te[P]_k(\tau)}_{L^2}]
    d\tau\\
    &+C\int_0^tV'(\tau)\gamma_k(\tau)2^{-k(s-1)}
    (\norm{\te[c](\tau)}_{\be[s-1]}
    +\norm{\te[\h](\tau)}_{\be[s-1]})d\tau.
\end{align*}
Therefore, changing the functions $(\te[\h],\te[c],\te[\inc], \te[F],\te[G], \te[P])$ into the original ones, we get
\begin{align*}
    &\sum_{k\ls0}2^{k(s-1)}\norm{\h_k(t)}_{\te[L]_T^\infty(L^2)}
    +\sum_{k\ls 0}2^{k(s-1)}\norm{c_k(t)}_{\te[L]_T^\infty(L^2)}\\
    &+\sum_{k\ls 0}2^{k(s-1)}\norm{\inc_k(t)}_{\te[L]_T^\infty(L^2)}+C\int_0^t \sum_{k\ls 0}2^{k(s+1)}\norm{c_k(\tau)}_{L^2}d\tau\\
    &    +C\int_0^t \sum_{k\ls 0}2^{k(s+1)}\norm{\inc_k(\tau)}_{L^2}d\tau\\
    \lesssim&e^{CV(t)}\norm{(\h(0),c(0),\inc(0))}_{\hbe{s}\times(\be[s-1])^2}
    +\int_0^t e^{CV(t-\tau)}\norm{(G(\tau),P(\tau))}_{(\be[s-1])^2}d\tau\\
    &+e^{CV(t)}
    (\norm{c}_{\te[L]_t^\infty(\be[s-1])}
    +\norm{\h(\tau)}_{\te[L]_T^\infty(\hbe[s-1]{s})})
\end{align*}
Using \eqref{swm63}, we eventually conclude that
\begin{align*}
    &C\int_0^t \sum_{k\ls 0}2^{k(s+1)}\norm{c_k(\tau)}_{L^2}d\tau
    +C\int_0^t \sum_{k\ls 0}2^{k(s+1)}\norm{\inc_k(\tau)}_{L^2}d\tau\\
    \lesssim &e^{CV(t)}\Big( \norm{(\h(0),c(0),\inc(0))}_{\hbe{s}
    \times(\be[s-1])^2}\\
    &\qquad\qquad+\int_0^t e^{-CV(\tau)}\norm{(F(\tau), G(\tau),P(\tau))}_{\hbe[s-1]{s}\times(\be[s-1])^2}d\tau\Big)
\end{align*}
Combining the last inequality with \eqref{swm63}, we complete the proof.
\end{proof}

\section{Existence and uniqueness}

This section is devoted to the proof of the Theorem~\ref{thm.main}.
The principle of the proof is a very classical one. We shall use the classical Friedrichs' regularization method, which was used in \cite{CheM01,CheZ07,CM08,Has08} for examples, to construct the
approximate solutions  $(\h^n,\vu^n)_{n\in\NN}$ to
\eqref{swm}, and then we will use Proposition~\ref{prop.conven} to get some uniform bounds on $(\h^n,\vu^n)_{n\in\NN}$.

\subsection{Construction of the approximate sequence} 

To this end, let us define the sequence of operators
$(J_n)_{n\in\NN}$ by
\begin{align*}
    J_n f:=\F^{-1}\mathbf{1}_{B(\frac{1}{n},n)}(\xi)\F f,
\end{align*}
and consider the following approximate system:
\begin{align}\label{swm6}
\left\{\begin{aligned}
    &\h_t^{n}+J_n(J_n\vu^n\cdot\nabla J_n\h^{n})+\hb\lam J_nc^{n}=F^n,\\
    &c_t^{n}+J_n(J_n\vu^n\cdot\nabla J_nc^{n})-4\mu\Delta J_nc^{n}-\f J_n\inc^{n}-\g\lam J_n\h^{n}
    -\beta\lam^3 J_n\h^{n}=G^n,\\
    &\inc_t^{n}-\mu\Delta J_n\inc^{n}+\f J_n c^{n}=J_n\lam^{-1}\divp H^n,\\
    &\vu^{n}=-\lam^{-1}\nabla c^{n}-\lam^{-1}\nabla^\perp\inc^{n},\\
    &(\h^{n},c^{n},\inc^{n})(0)=(\h_n,\lam^{-1}\dive\vu_n,
    \lam^{-1}\divp\vu_n),
    \end{aligned}\right.
\end{align}
where
\begin{align*}
    & \h_n=J_n(\h_0-\hb), \quad
    \vu_n=J_n\vu_0,\\
    &F^n=-J_n(J_n\h^n\dive J_n\vu^n),\\
    &G^n=J_n(J_n\vu^n\cdot\nabla J_nc^n)+J_n\lam^{-1}\dive H^n,\\
    &H^n=-J_n\vu^n\cdot\nabla J_n\vu^n+2\mu\frac{\nabla J_n\h^n D(J_n\vu^n)
    +\nabla J_n\h^n\dive J_n\vu^n}{ \zeta(J_n\h^n+\hb)},
\end{align*}
with $\zeta$ a smooth function satisfying
\begin{align*}
    \zeta(s)=\left\{\begin{array}{ll}
        \hb/4,&\abs{s}\ls\hb/4,\\
        s, &\hb/2\ls\abs{s}\ls 3\hb/2,\\
        7\hb/4, &\abs{s}\gs 7\hb/4.\\
        \text{smooth}, &\text{otherwise}.
    \end{array}\right.
\end{align*}
We want to show that \eqref{swm6} is only an ordinary differential equation in $L^2\times L^2\times L^2$. We can observe easily that all the source term in \eqref{swm6} turn out to be continuous in $L^2\times L^2\times L^2$. For example, we consider the term $J_n\lam^{-1}\dive \frac{\nabla J_n\h^n\dive J_n\vu^n}{\zeta(J_n\h^n+\hb)}$. By Plancherel's theorem, Hausdorff-Young's inequality and H\"older's inequality, we have
\begin{align*}
    &\norm{J_n\lam^{-1}\dive \frac{\nabla J_n\h^n\dive J_n\vu^n}{\zeta(J_n\h^n+\hb)}}_{L^2}
    =\norm{\mathbf{1}_{B(\frac{1}{n},n)}\abs{\xi}^{-1}(\xi_1,\xi_2)\cdot
    \F\frac{\nabla J_n\h^n\dive J_n\vu^n}{\zeta(J_n\h^n+\hb)}}_{L^2}\\
    \ls &\norm{\frac{\nabla J_n\h^n\dive J_n\vu^n}{\zeta(J_n\h^n+\hb)}}_{L^2}
    \ls \norm{\nabla J_n\h^n\dive J_n\vu^n}_{L^2}\norm{\frac{1}{\zeta(J_n\h^n+\hb)}}_{L^\infty}\\
    \ls &\frac{4}{\hb}\norm{\nabla J_n\h^n}_{L^\infty}\norm{\dive J_n\vu^n}_{L^2}
    \ls \frac{4n}{\hb}\norm{\abs{\xi}\mathbf{1}_{B(\frac{1}{n},n)}\F \h^n}_{L^1}\norm{\vu^n}_{L^2}\\
    \ls &\frac{4n^3}{\hb}\norm{\h^n}_{L^2}\norm{\vu^n}_{L^2}.
\end{align*}
Thus, the usual Cauchy-Lipschitz theorem implies the existence of a strictly positive maximal time $T_n$ such that a unique solution exists which is continuous in time with value in $L^2\times L^2\times L^2$. However, as $J_n^2=J_n$, we claim that $J_n(\h^n,c^n,\inc^n)$ is also a solution, so uniqueness implies that $J_n(\h^n,c^n,\inc^n)=(\h^n,c^n,\inc^n)$. So $(\h^n,c^n,\inc^n)$ is also a solution of the following system:
\begin{align}\label{swm8}
\left\{\begin{aligned}
    &\h_t^{n}+J_n(\vu^n\cdot\nabla \h^{n})+\hb\lam c^{n}=F_1^n,\\
    &c_t^{n}+J_n(\vu^n\cdot\nabla c^{n})-4\mu\Delta c^{n}-\f \inc^{n}-\g\lam \h^{n}
    -\beta\lam^3 \h^{n}=G_1^n,\\
    &\inc_t^{n}-\mu\Delta \inc^{n}+\f c^{n}=\lam^{-1}\divp H_1^n,\\
    &\vu^{n}=-\lam^{-1}\nabla c^{n}-\lam^{-1}\nabla^\perp\inc^{n},\\
    &(\h^{n},c^{n},\inc^{n})(0)=(\h_n,\lam^{-1}\dive\vu_n,
    \lam^{-1}\divp\vu_n),
    \end{aligned}\right.
\end{align}
with
\begin{align*}
    & \h_n=J_n(\h_0-\hb), \quad
    \vu_n=J_n\vu_0,\\
    &F_1^n=-J_n(\h^n\dive \vu^n),\\
    &G_1^n=J_n(\vu^n\cdot\nabla c^n)+J_n\lam^{-1}\dive H_1^n,\\
    &H_1^n=-\vu^n\cdot\nabla \vu^n+2\mu\frac{\nabla \h^n D(\vu^n)
    +\nabla \h^n\dive \vu^n}{ \zeta(\h^n+\hb)}.
\end{align*}
The system \eqref{swm8} appears to be an ordinary differential equation in the space
\begin{align*}
    L_n^2:=\set{a\in L^2(\Real^2): \supp\F a\subset B(\frac{1}{n},n)}.
\end{align*}
Due to the Cauchy-Lipschitz theorem again, a unique maximal solution exists on an interval $[0,T_n^*)$ which is continuous in time with value in $L_n^2\times L_n^2\times L_n^2$.

\subsection{Uniform bounds}
In this part, we prove uniform estimates independent of $T<T_n^*$ in $E_T^1\cap E_T^{1+\eps}$ for
$(\h^n,\vu^n)$. We shall show that $T_n^*=+\infty$ by the Cauchy-Lipschitz theorem.
Denote
\begin{align*}
    E(0):=&\norm{\h_0-\hb}_{\hbe[0]{1+\eps}}
    +\norm{\vu_0}_{\hbe[0]{\eps}},\\
    E(\h,\vu,t):=&\norm{(\h,\vu)}_{E_t^1}
    +\norm{(\h,\vu)}_{E_t^{1+\eps}},\\
    \te[T]_n:=&\sup\set{t\in[0,T_n^*): E(\h^n,\vu^n,t)\ls A\te[C]E(0)},
\end{align*}
where $\te[C]$ corresponds to the constant in Proposition~\ref{prop.conven} and $A>\max(2,\te[C]^{-1})$ is a constant. Thus, by the continuity we have $\te[T]_n>0$.

We are going to prove that $\te[T]_n=T_n^*$ for all $n\in\NN$ and we will conclude that $T_n^*=+\infty$ for any $n\in\NN$.

According to the Proposition~\ref{prop.conven}
and the definition of $(\h_n,\vu_n)$, the following inequality
holds
\begin{align*}
    \norm{(\h^{n},\vu^{n})}_{E_T^1}\ls
    &\te[C]e^{\te[C]\norm{\vu^n}_{L_T^1(\be[2])}}
    \Big(\norm{\h_0-\hb}_{\hbe[0]{1}}+\norm{\vu_0}_{\be[0]}\\
    &+\norm{F_1^n}_{L_T^1(\hbe[0]{1})}+\norm{\vu^n\cdot\nabla c^n}_{L_T^1(\be[0])}
    +\norm{H_1^n}_{L_T^1(\be[0])}\Big).
\end{align*}

Therefore, it is only a matter to prove appropriate estimates for
$F_1^n$, $H_1^n$ and $\vu^n\cdot\nabla c^n$. The estimate of $F_1^n$ is straightforward. From Lemma~\ref{lem.fgbesov}, we have
\begin{align}\label{uniest2}
    \norm{F_1^n}_{L_T^1(\hbe[0]{1})}\ls
    C\norm{\h^n}_{L_T^\infty(\hbe[0]{1})}
    \norm{\vu^n}_{L_T^1(\be[2])}
    \ls CE^2(\h^n,\vu^n,T).
\end{align}
With the help of Lemma~\ref{lem.fgbesov} and interpolation
arguments, we have
\begin{align}\label{uniest3}
\begin{aligned}
    \norm{\vu^n\cdot\nabla
    c^n}_{L_T^1(\be[0])}\ls&C\norm{\vu^n}_{L_T^2(\be[1])}\norm{\nabla
    c^n}_{L_T^2(\be[0])}\ls C\norm{\vu^n}_{L_T^2(\be[1])}^2\\
    \ls& C
    \norm{\vu^n}_{L_T^\infty(\be[0])}\norm{\vu^n}_{L_T^1(\be[2])}\\
    \ls & CE^2(\h^n,\vu^n,T).
    \end{aligned}
\end{align}
In the same way, we can get
\begin{align}\label{uniest4}
\norm{\vu^n\cdot\nabla \vu^n}_{L_T^1(\be[0])}\ls CE^2(\h^n,\vu^n,T).
\end{align}
To estimate other terms of $H_1^n$, we make the following assumption
on $E(0)$:
\begin{align*}
    2C_1A\te[C]E(0)\ls \hb,
\end{align*}
where $C_1$ is the continuity modulus of $\be[1]\subset L^\infty$.
If $T< \te[T]_n$, it implies
\begin{align*}
    \norm{\h^n}_{L^\infty}\ls C_1\norm{\h^n}_{\be[1]}\ls
    C_1\norm{\h^n}_{\hbe[0]{1}}\ls C_1 A\te[C]E(0)\ls\frac{1}{2}\hb.
\end{align*}
Thus,  we have
\begin{align*}
    \norm{\h^n}_{L^\infty([0,T]\times\Real^2)}\ls \frac{1}{2}\hb,
\end{align*}
which yields
\begin{align*}
    \h^n+\hb\in[\frac{1}{2}\hb,\frac{3}{2}\hb] \text{ and }
    \zeta(\h^n+\hb)=\h^n+\hb.
\end{align*}

 From Lemma~\ref{lem.fgbesov} and \ref{lem.comp}, and interpolation arguments, we have
\begin{align}\label{uniest5}
    \begin{aligned}
    &\hb\norm{\frac{\nabla\h^n\cdot\nabla
    \vu^n}{\hb+\h^n}}_{L_T^1(\be[0])}\\
    \ls& \norm{\nabla\h^n\cdot\nabla \vu^n}_{L_T^1(\be[0])}
    +\norm{\frac{\h^n\nabla\h^n\cdot\nabla
    \vu^n}{\hb+\h^n}}_{L_T^1(\be[0])}\\
    \ls &C
    \norm{\nabla\h^n}_{L_T^\infty(\be[0])}
    \norm{\nabla\vu^n}_{L_T^1(\be[1])}
    +C\norm{\frac{\h^n\nabla\h^n}{\hb+\h^n}}_{L_T^\infty(\be[0])}
    \norm{\nabla\vu^n}_{L_T^1(\be[1])}\\
    \ls
    &C\norm{\h^n}_{L_T^\infty(\hbe[0]{1})}\norm{\vu^n}_{L_T^1(\be[2])}
    \Big(1+\norm{\frac{\h^n}{\hb+\h^n}}_{L_T^\infty(\be[1])}\Big)\\
    \ls
    &C\norm{\h^n}_{L_T^\infty(\hbe[0]{1})}\norm{\vu^n}_{L_T^1(\be[2])}
    \Big(1+\norm{\h^n}_{L_T^\infty(\be[1])}\Big)\\
    \ls& CE^2(\h^n,\vu^n,T)(1+E(\h^n,\vu^n,T)).
    \end{aligned}
\end{align}
Similarly, we can get
\begin{align}\label{uniest51}
    \hb\norm{\frac{\nabla\h^n\cdot D(\vu^n)}{\hb+\h^n}}_{L_T^1(\be[0])} \ls CE^2(\h^n,\vu^n,T)(1+E(\h^n,\vu^n,T)).
\end{align}
Hence, from \eqref{uniest3}-\eqref{uniest5}, we gather
\begin{align}\label{uniest6}
\begin{aligned}
   & \norm{\vu^n\cdot\nabla
    c^n}_{L_T^1(\be[0])}+\norm{H_1^n}_{L_T^1(\be[0])}\\
    \ls &
    C(1+4\mu\hb^{-1}(1+E(\h^n,\vu^n,T)))E^2(\h^n,\vu^n,T).
\end{aligned}
\end{align}

Similarly, according to the Proposition~\ref{prop.conven}
and the definition of $(\h_n,\vu_n)$, the following inequality
holds
\begin{align*}
    \norm{(\h^{n},\vu^{n})}_{E_T^{1+\eps}}\ls
    &\te[C]e^{\te[C]\norm{\vu^n}_{L_T^1(\be[2])}}
    \Big(\norm{\h_0-\hb}_{\hbe[\eps]{1+\eps}}+\norm{\vu_0}_{\be[\eps]}\\
    &+\norm{F_1^n}_{L_T^1(\hbe[\eps]{1+\eps})}+\norm{\vu^n\cdot\nabla c^n}_{L_T^1(\be[\eps])}
    +\norm{H_1^n}_{L_T^1(\be[\eps])}\Big).
\end{align*}

The estimate of $F_1^n$
is straightforward. From Lemma~\ref{lem.fgbesov}, we have
\begin{align}\label{uniest8}
    \norm{F_1^n}_{L_T^1(\hbe[\eps]{1+\eps})}\ls
    C\norm{\h^n}_{L_T^\infty(\hbe[\eps]{1+\eps})}
    \norm{\vu^n}_{L_T^1(\be[2])}
    \ls CE^2(\h^n,\vu^n,T).
\end{align}
With the help of Lemma~\ref{lem.fgbesov} and interpolation
arguments, we have
\begin{align}\label{uniest9}
\begin{aligned}
    \norm{\vu^n\cdot\nabla
    c^n}_{L_T^1(\be[\eps])}
    \ls&C\norm{\vu^n}_{L_T^{2+\eps}(\be[1])}\norm{\nabla
    c^n}_{L_T^{\frac{2+\eps}{1+\eps}}(\be[\eps])}
    \\
    \ls& C
    \norm{\vu^n}_{L_T^\infty(\be[0])}\norm{\vu^n}_{L_T^1(\be[2+\eps])}\\
    \ls & CE^2(\h^n,\vu^n,T).
    \end{aligned}
\end{align}
In the same way, we can get
\begin{align}\label{uniest10}
\norm{\vu^n\cdot\nabla \vu^n}_{L_T^1(\be[\eps])}\ls CE^2(\h^n,\vu^n,T).
\end{align}

From Lemma~\ref{lem.fgbesov}, Lemma~\ref{lem.comp} and interpolation arguments, we have
\begin{align}\label{uniest11}
    \begin{aligned}
    &\hb\norm{\frac{\nabla\h^n\cdot\nabla
    \vu^n}{\hb+\h^n}}_{L_T^1(\be[\eps])}\\
    \ls& \norm{\nabla\h^n\cdot\nabla \vu^n}_{L_T^1(\be[\eps])}
    +\norm{\frac{\h^n\nabla\h^n\cdot\nabla
    \vu^n}{\hb+\h^n}}_{L_T^1(\be[\eps])}\\
    \ls &C
    \norm{\nabla\h^n}_{L_T^\infty(\be[\eps])}
    \norm{\nabla\vu^n}_{L_T^1(\be[1])}
    +C\norm{\frac{\h^n\nabla\h^n}{\hb+\h^n}}_{L_T^\infty(\be[\eps])}
    \norm{\nabla\vu^n}_{L_T^1(\be[1])}\\
    \ls
    &C\norm{\h^n}_{L_T^\infty(\hbe[0]{1+\eps})}
    \norm{\vu^n}_{L_T^1(\be[2])}
    \Big(1+\norm{\frac{\h^n}{\hb+\h^n}}_{L_T^\infty(\be[1])}\Big)\\
    \ls
    &C\norm{\h^n}_{L_T^\infty(\hbe[0]{1+\eps})}
    \norm{\vu^n}_{L_T^1(\be[2])}\Big(1+\norm{\h^n}_{L_T^\infty(\be[1])}\Big)\\
    \ls& CE^2(\h^n,\vu^n,T)(1+E(\h^n,\vu^n,T)).
    \end{aligned}
\end{align}
Similarly, we can get
\begin{align}\label{uniest12}
    \hb\norm{\frac{\nabla\h^n\cdot D(\vu^n)}{\hb+\h^n}}_{L^1(\be[\eps])} \ls CE^2(\h^n,\vu^n,T)(1+E(\h^n,\vu^n,T)).
\end{align}
Hence, from \eqref{uniest9}-\eqref{uniest12}, we gather
\begin{align}\label{uniest13}
\begin{aligned}
   & \norm{\vu^n\cdot\nabla
    c^n}_{L_T^1(\be[\eps])}+\norm{H_1^n}_{L_T^1(\be[\eps])}\\
    \ls &
    C(1+4\mu\hb^{-1}(1+E(\h^n,\vu^n,T)))E^2(\h^n,\vu^n,T).
\end{aligned}
\end{align}

From \eqref{uniest2}, \eqref{uniest6}, \eqref{uniest8} and \eqref{uniest13}, it follows
\begin{align*}
    \norm{(\h^{n},\vu^{n})}_{E_T^1\cap E_T^{1+\eps}}\ls \te[C]e^{A\te[C]^2E(0)}[
    1+CA^2\te[C]^2(1+4\mu\hb^{-1}(1+A\te[C]E(0)))E(0)]E(0).
\end{align*}

So we can choose $E(0)$ so small that
\begin{align}\label{uniest7}
\begin{aligned}
    &1+CA^2\te[C]^2(1+4\mu\hb^{-1}(1+A\te[C]E(0)))E(0)\ls \frac{A^2}{A+2},\\
    & e^{A\te[C]^2E(0)}\ls \frac{A+1}{A}\; \text{
    and }\;
    2C_1A\te[C]E(0)\ls \hb,
\end{aligned}
\end{align}
which yields $\norm{(\h^{n},\vu^{n})}_{E_T^1}\ls \frac{A+1}{A+2}A\te[C]E(0)$ for any $T<\te[T]_n$. It follows that $\te[T]_n=T_n^*$. In fact, if $\te[T]_n<T_n^*$, we have seen that $E(\h^n,\vu^n,\te[T]_n)\ls \frac{A+1}{A+2}A\te[C]E(0)$. So by continuity, for a sufficiently small constant $\sigma>0$ we can obtain $E(\h^n,\vu^n,\te[T]_n+\sigma)\ls A\te[C]E(0)$. This yields a contradiction with the definition of $\te[T]_n$.

Now, if $\te[T]_n=T_n^*<\infty$, then we have obtained $F(\h^n,\vu^n,T_n^*)\ls A\te[C]E(0)$. As $\norm{\h^n}_{L_{T_n^*}(\hbe[0]{1+\eps})}<\infty$ and $\norm{\vu^n}_{L_{T_n^*}(\hbe[0]{\eps})}<\infty$, it implies that $\norm{\h^n}_{L_{T_n^*}(L_n^2)}<\infty$ and $\norm{\vu^n}_{L_{T_n^*}(L_n^2)}<\infty$. Thus, we may continue the solution beyond $T_n^*$ by the Cauchy-Lipschitz theorem. This contradicts the definition od $T_n^*$. Therefore, the approximate solution $(\h^n,\vu^n)_{n\in\NN}$ is global in time.

\subsection{Existence of a solution}

In this part, we shall show that, up to an extraction, the sequence
$(\h^n,\vu^n)_{n\in\NN}$ converges in
$\mathscr{D}'(\Real^+\times\Real^2)$ to a solution $(\h,\vu)$ of
\eqref{swm} which has the desired regularity properties. The proof
lies on compactness arguments. To start with, we show that the time
first derivative of $(\h^n,\vu^n)$ is uniformly bounded in
appropriate spaces. This enables us to apply Ascoli's theorem and
get the existence of a limit $(\h,\vu)$ for a subsequence. Now, the
uniform bounds of the previous part provides us with additional
regularity and convergence properties so that we may pass to the
limit in the system.

It is convenient to split $(\h^n,\vu^n)$ into the solution of a
linear system with initial data $(\h_n,\vu_n)$ and the discrepancy
to that solution. More precisely, we denote by $(\h_L^n,\vu_L^n)$
the solution to the linear system
\begin{align}\label{linear}
\left\{\begin{aligned}
    &\partial_t\h_L^n+\dive\vu_L^n=0,\\
    &\partial_t\vu_L^n-\mu\Delta\vu_L^n-3\mu\nabla\dive\vu_L^n+\f (\vu_L^n)^\perp+\g\nabla\h_L^n
    -\beta\nabla\Delta\h_L^n=0,\\
    &(\h_L^n,\vu_L^n)_{t=0}=(\h_n,\vu_n),
    \end{aligned}\right.
\end{align}
and $(\bar{\h}^n,\bar{\vu}^n)=(\h^n-\h_L^n,\vu^n-\vu_L^n)$.

Obviously, the definition of $(\h_n,\vu_n)$ entails
\begin{align*}
    \h_n\to \h_0-\hb \text{ in } \hbe[0]{1+\eps}, \quad
    \vu_n\to\vu_0 \text{ in } \hbe[0]{\eps}, \quad \text{as } n\to +\infty.
\end{align*}
The Proposition~\ref{prop.conven} insures us
that
\begin{align}\label{linear1}
    (\h_L^n,\vu_L^n) \to (\h_L,\vu_L) \text{ in } E^1\cap E^{1+\eps},
\end{align}
where $(\h_L,\vu_L)$ is the solution of the linear system
\begin{align}\label{linear2}
\left\{\begin{aligned}
    &\partial_t\h_L+\dive\vu_L=0,\\
    &\partial_t\vu_L-\mu\Delta\vu_L-3\mu\nabla\dive\vu_L+\f\vu_L^\perp
    +\g\nabla\h_L-\beta\nabla\Delta\h_L=0,\\
    &(\h_L,\vu_L)_{t=0}=(\h_0-\hb,\vu_0).
    \end{aligned}\right.
\end{align}
Now, we have to prove the convergence of $(\bar{\h}^n,\bar{\vu}^n)$.
This is of course a trifle more difficult and requires compactness
results. Let us first state the following lemma.

\begin{lemma}\label{lem.ubt}
$((\bar{\h}^n,\bar{\vu}^n))_{n\in\NN}$ is uniformly bounded in  $\C^{\frac{1}{2}}(\Real^+; \be[0]) \times
(\C^{\frac{\eps}{3+\eps}}(\Real^+; \be[0]))^2$.
\end{lemma}

\begin{proof}
Throughout the proof, we will note u.b. for uniformly bounded. We
first prove that $\partial_t \bar{\h}^n$ is u.b. in
$L^2(\Real^+,\be[0])$, which yields the desired result for
$\bar{\h}$. Let us observe that $\bar{\h}^n$ verifies the following
equation
\begin{align*}
    \partial_t\bar{\h}^{n}=-J_n(\h^n\dive\vu^n)
    -J_n(\vu^n\cdot\nabla\h^{n})-\hb\dive\vu^{n}+\hb\dive\vu_L^{n}.
\end{align*}
According to the previous part, $(\h^n)_{n\in\NN}$ is u.b. in
$\te[L]^\infty(\be[1])$ and $(\vu^n)_{n\in\NN}$ is u.b. in $\te[L]^2(\be[1])$
in view of interpolation arguments. Thus, $-J_n(\h^n\dive\vu^n)
-J_n(\vu^n\cdot\nabla\h^{n})-\dive\vu^{n}$ is u.b. in $\te[L]^2(\be[0])$.
The definition of $\vu_L^n$ obviously provides us with uniform
bounds for $\dive\vu_L^n$ in $\te[L]^2(\be[0])$, so we can conclude that
$\partial_t\bar{\h}^n$ is u.b. in $L^2(\be[0])$.

Denote $c_L^n=\lam^{-1}\dive\vu_L^n$, $\bar{c}^n=
\lam^{-1}\dive\bar{\vu}^n$, $\inc_L^n= \lam^{-1}\dive^\perp\vu_L^n$ and
$\bar{\inc}^n= \lam^{-1}\dive^\perp\bar{\vu}^n$. Let us prove now that
$\partial_t\bar{c}^n$ is u.b. in $(L^{\frac{3+\eps}{3}}+L^\infty)(\be[0])$ and that $\partial_t\bar{\inc}^n$ is u.b.
in $(L^{\frac{2+\eps}{2}}+L^\infty)(\be[0])$ which give the required
result for $\bar{\vu}^n$ by using the relation
$\vu^n=-\lam^{-1}\nabla c^n-\lam^{-1}\nabla^\perp\inc^n$.

Let us recall that
\begin{align*}
    \partial_t\bar{c}^{n}=&
    4\mu\Delta(c^{n}-c_L^{n})+\f(\inc^{n}-\inc_L^{n})
    +\g\lam(\h^{n}-\h_L^{n})+\beta\lam^3(\h^{n}-\h_L^{n})\\
    &-J_n\lam^{-1}\dive\left(\vu^n\cdot\nabla\vu^n-2\mu\frac{\nabla\h^n D(\vu^n) +\nabla\h^n\dive\vu^n}{\h^n+\hb}\right),\\
    \partial_t\bar{\inc}^{n}=&\mu\Delta(\inc^{n}-\inc_L^{n})
    -\f(c^{n}-c_L^{n})\\
    &-J_n\lam^{-1}\divp\left(\vu^n\cdot\nabla\vu^n-2\mu\frac{\nabla\h^n D(\vu^n)
    +\nabla\h^n\dive\vu^n}{\h^n+\hb}\right).
\end{align*}

Results of the previous part and an interpolation argument yield uniform bounds for $\vu^n$ in $\te[L]^{\frac{2+\eps}{\eps}}(\be[\eps])\cap \te[L]^{\frac{2+\eps}{2-\eps}}
(\be[2-\eps])$. Since $\h^n$ is u.b. in $\te[L]^\infty(\be[1])$,
$c_L^n$ and $c^n$ are u.b. in $\te[L]^{\frac{2+\eps}{2}} (\be[2])$, we easily
verify that $\Delta(c^{n}-c_L^{n})$ and
$J_n\lam^{-1}\dive\Big(\vu^n\cdot\nabla\vu^n-2\mu\frac{\nabla\h^n D(\vu^n) +\nabla\h^n\dive\vu^n}{\h^n+\hb}\Big)$ are u.b. in
$L^{\frac{2+\eps}{2}}(\be[0])$. Obviously, we have $\inc^{n+1}$ and $\inc_L^{n+1}$ u.b. in $\te[L]^\infty(\be[0])$. Because $\h^n$ and $\h_L^n$ are u.b. in $\te[L]^\infty(\be[1])$, we have $\lam(\h^{n+1}-\h_L^{n+1})$ u.b. in $L^\infty(\be[0])$.
We also have $\h^n$ and $\h_L^n$ are u.b. in
$(\te[L]^{\frac{4}{3}}+\te[L]^{\frac{3+\eps}{3}})(\be[3])$ in view of Lemma~\ref{lem.inter}. Thus, $\lam^3\h^n$ is u.b. in
$(L^{\frac{4}{3}}+L^{\frac{3+\eps}{3}})(\be[0])$.  So we finally get $\partial_t\bar{c}^n$
u.b. in $(L^{\frac{3+\eps}{3}}+L^\infty)(\be[0])$. The case of $\partial_t\bar{\inc}^n$ goes
along the same lines. As the terms corresponding to
$\lam^3(\h^n-\h_L^n)$ do not appear, we simply get
$\partial_t\bar{\inc}^n$ u.b. in
$(L^{\frac{2+\eps}{2}}+L^\infty)(\be[0])$.
\end{proof}

Now, we can turn to the proof of the existence of a solution and use
Ascoli theorem to get strong convergence. We need to localize the
spatial space because we have some results of compactness for the
local Sobolev spaces. Let $(\chi_p)_{p\in\NN}$ be a sequence of
$\C_0^\infty(\Real^2)$ cut-off functions supported in the ball
$B(0,p+1)$ of $\Real^2$ and equal to $1$ in a neighborhood of
$B(0,p)$.

For any $p\in\NN$, Lemma~\ref{lem.ubt} tells us that
$((\chi_p\bar{\h}^n, \chi_p\bar{\vu}^n))_{n\in\NN}$ is uniformly
equicontinuous in $\C(\Real^+;(\be[0])^{1+2})$ and bounded in $E^{1+\eps}$.

Let us observe that the application $f\mapsto \chi_p f$ is compact from $\hbe[0]{1}$ into $\be[0]$, and from
$\be[\eps]$ into $\hbe[\eps]{0}$. After we apply Ascoli's theorem to the
family $((\chi_p\bar{\h}^n, \chi_p\bar{\vu}^n))_{n\in\NN}$ on the
time interval $[0,p]$, we use Cantor's diagonal process. This
finally provides us with a distribution $(\bar{\h},\bar{\vu})$
belonging to $\C(\Real^+; \be[0]\times
(\hbe[\eps]{0})^2)$ and a subsequence (which we still denote
by $((\bar{\h}^n,\bar{\vu}^n)_{n\in\NN})$ such that, for all
$p\in\NN$, we have
\begin{align*}
    (\chi_p\bar{\h}^n,\chi_p\bar{\vu}^n)\to
    (\chi_p\bar{\h},\chi_p\bar{\vu}) \text{ as } n\to+\infty,
\end{align*}
 in  $\C([0,p]; \be[0]\times (\hbe[\eps]{0})^2)$. This obviously infers that $(\bar{\h}^n,\bar{\vu}^n)$ tends to
$(\bar{\h},\bar{\vu})$ in $\mathscr{D}'(\Real^+\times \Real^2)$.

Coming back to the uniform estimates of the previous part, we
moreover get that $(\bar{\h},\bar{\vu})$ belongs to
\begin{align}\label{linear3}
    \te[L]^\infty\left(\Real^+;\hbe[0]{1+\eps}\times(\hbe[0]{\eps})^2\right)
    \cap L^1\left(\Real^+; (\hbe[4]{3}\cap\hbe[4+\eps]{3+\eps})\times
    (\hbe[2]{2+\eps})^2\right)
\end{align}
and to $\C^{1/2}(\Real^+;\be[0])\times
(\C^{\frac{\eps}{3+\eps}}(\Real^+;\be[0]))^2$.

Let us now prove that $(\h,\vu):=(\h_L,\vu_L)+(\bar{\h},\bar{\vu})$
solves \eqref{swm}. We first observe that, according to
\eqref{swm6},
\begin{align*}
    \left\{\begin{aligned}
      &\h_t^{n}+J_n(\vu^n\cdot\nabla\h^{n})+\hb\lam c^{n}=-J_n(\h^n\dive\vu^n),\\
      &\vu_t^{n}+J_n(\vu^n\cdot\nabla\vu^n)-\mu\Delta\vu^{n}-3\mu\nabla\dive\vu^{n}
      +\f(\vu^{n})^\perp+\g\nabla\h^{n}-\beta\nabla\Delta\h^{n}\\
      &=2\mu J_n\frac{\nabla\h^n D(\vu^n)
    +\nabla\h^n\dive\vu^n}{\h^n+\hb}.
    \end{aligned}\right.
\end{align*}
The only problem is to pass to the limit in
$\mathscr{D}'(\Real^+\times\Real^2)$ in the nonlinear terms. This
can be done by using the convergence results stemming from the
uniform estimates and the convergence results \eqref{linear1} and
\eqref{linear3}.

As it is just a matter of doing tedious verifications, we show, as
an example, the case of the term
$\frac{\nabla\h^n\dive\vu^n}{\hb+\h^n}$. Denote $L(z)=z/(z+\hb)$.
Let $\theta\in\C_0^\infty(\Real^+\times\Real^2)$ and $p\in\NN$ be
such that $\supp\theta\subset [0,p]\times B(0,p)$. We consider the
decomposition
\begin{align*}
    &J_n\frac{\hb\theta\nabla\h^n\dive\vu^n}{\hb+\h^n}
    -\frac{\hb\theta\nabla\h\dive\vu}{\hb+\h}\\
    =&J_n[\theta(1-L(\h^n))\chi_p\nabla\h^n\chi_p
    \dive(\vu_L^n-\vu_L)+\theta(1-L(\h^n))
    \chi_p\nabla\h^n\chi_p\dive(\chi_p(\bar{\vu}^n-\bar{\vu}))\\
    & +\theta(1-L(\h^n))\chi_p\nabla
    (\chi_p(\h^n-\h))\dive\vu
    +\theta\nabla\h\chi_p\dive\vu(L(\chi_p\h)-L(\chi_p\h^n))]\\
    &+(J_n-I)\frac{\hb\theta\nabla\h\dive\vu}{\hb+\h}.
\end{align*}
The last term tends to zero as $n\to +\infty$ due to the property of $J_n$. As $\theta L(\h^n)$ and $\h^n$ are u.b. in $L^\infty(\be[1])$ and
$\vu_L^n$ tends to $\vu_L$ in $L^1(\be[2])$, the first term tends to
$0$ in $L^1(\be[0])$. According to \eqref{linear3},
$\chi_p(\bar{\vu}^n-\bar{\vu})$ tends to zero in
$L^1([0,p];\be[2])$ so that the second term tends to $0$ in
$L^1([0,p];\be[0])$. Clearly, $\chi_p\h^n\to\chi_p\h$ in
$L^\infty(\be[1])$ and $L(\chi_p\h^n)\to L(\chi_p\h)$ in
$L^\infty(L^\infty\cap\be[1])$, so that the third and the last
terms also tend to $0$ in $L^1(\be[0])$. The other nonlinear terms can
be treated in the same way.

We still have to prove that $\h$ is continuous in $\hbe[0]{1+\eps}$
and that $\vu$ belongs to $\C(\Real^+;(\hbe[0]{\eps})^2)$. The continuity of
$\vu$ is straightforward. Indeed, $\vu$ satisfies
\begin{align*}
    \partial_t\vu=&-\vu\cdot\nabla\vu+\mu\Delta\vu+3\mu\nabla\dive\vu
    -\f\vu^\perp-\g\nabla\h+\beta\nabla\Delta\h\\
    &+2\mu\frac{\nabla\h D(\vu)+\nabla\h\dive\vu}{\h+\hb}
\end{align*}
and the r.h.s. belongs to $(L^1+L^\infty)(\be[0])$ by noting that we also have
$\h\in L^\infty(\be[1])\cap (L^{\frac{4}{3}}+L^1)(\be[3])$ in view of the interpolation
argument.  In a similar argument, one can obtain $\vu\in \C(\Real^+;(\be[\eps])^2)$. We have already got that $\h\in\C(\Real^+;\be[0])$.
Indeed, $\h_0-\hb\in \be[0]$, $\vu\in L^2(\Real^+;\be[1])$, $\h\in
L^\infty(\Real^+;\be[1])$ and then $\partial_t\h\in
L^2(\Real^+;\be[0])$ from the equation
$\partial_t\h=-\hb\dive\vu-\dive(\h\vu)$. Thus, there remains to prove
the continuity of $\h$ in $\be[1+\eps]$.

Let us apply the operator $\dk$ to the first equation of \eqref{swm}
to get
\begin{align}\label{linear4}
    \partial_t\dk\h=-\dk(\vu\cdot\nabla\h)-\hb\dk\dive\vu
    -\dk(\h\dive\vu).
\end{align}
Obviously, for fixed $k$ the r.h.s belongs to $L_{loc}^1(\Real^+;
L^2)$ so that each $\dk\h$ is continuous in time with values in
$L^2$.

Now, we apply an energy method to \eqref{linear4} to obtain, with
the help of Lemma~6.2 in \cite{Dan01}, that
\begin{align*}
    \frac{1}{2}\frac{d}{dt}\norm{\dk\h}_{L^2}^2\ls
    C\norm{\dk\h}_{L^2}\Big(\alpha_k2^{-k(1+\eps)}\norm{\h}_{\be[1+\eps]}
    \norm{\vu}_{\be[2]}\\+\norm{\dk\dive\vu}_{L^2}
    +\norm{\dk(\h\dive\vu)}_{L^2}\Big),
\end{align*}
where $\sum_k\alpha_k\ls 1$. Integrating in time and multiplying
$2^{k(1+\eps)}$, we get
\begin{align*}
    &2^{k(1+\eps)}\norm{\dk\h(t)}_{L^2}\\
    \ls& 2^{k(1+\eps)}\norm{\dk(\h_0-\hb)}_{L^2}
    +C\int_0^t
    \Big(\alpha_k\norm{\h(\tau)}_{\be[1+\eps]}\norm{\vu(\tau)}_{\be[2]}\\
    &+2^{k(2+\eps)}\norm{\dk\vu(\tau)}_{L^2}
    +2^{k(1+\eps)}\norm{\dk(\h\dive\vu)(\tau)}_{L^2}\Big)d\tau.
\end{align*}
Since $\h\in L^\infty(\be[1+\eps])$, $\vu\in L^1(\hbe[2]{2+\eps})$ and
$\h\dive\vu\in L^1(\be[1+\eps])$, we can get
\begin{align*}
    \sum_{k\in\Z} \sup_{t\gs 0} 2^{k(1+\eps)}\norm{\dk\h(t)}_{L^2}
    \lesssim&
    \norm{\h_0-\hb}_{\be[1+\eps]}
    +\left(1+\norm{\h}_{L^\infty(\be[1+\eps])}\right)
    \norm{\vu}_{L^1(\hbe[2]{2+\eps})}\\
    &+\norm{\h\dive\vu}_{L^1(\be[1+\eps])}<\infty.
\end{align*}
Thus, $\sum_{\abs{k}\ls N}\dk\h$ converges uniformly in
$L^\infty(\Real^+;\be[1+\eps])$ and we can conclude that $\h\in
\C(\Real^+;\be[1+\eps])$.

\subsection{Uniqueness}

Let $(\h_1,\vu_1)$ and $(\h_2,\vu_2)$ be solutions of \eqref{swm} in
$E_T^1\cap E_T^{1+\eps}$ with the same data $(\h_0-\hb,\vu_0)$ constructed in the
previous parts on the time interval $[0,T]$. Denote
$(\delta\h,\delta\vu)=(\h_2-\h_1,\vu_2-\vu_1)$. From \eqref{swm}, we can get
\begin{align}\label{uni}
    \left\{\begin{aligned}
     &\partial_t\delta\h+\vu_2\cdot\nabla\delta\h+\hb\dive\delta\vu=F_2,\\
     &\partial_t\delta\vu+\vu_2\cdot\nabla\delta\vu
     -\mu\Delta\delta\vu-3\mu\nabla\dive\delta\vu+\f(\delta\vu)^\perp
     +\g\nabla\delta\h-\beta\nabla\Delta\delta\h=G_2,\\
     &(\delta\h,\delta\vu)=(0,\mathbf{0}),
    \end{aligned}
    \right.
\end{align}
where
\begin{align*}
    F_2=&-\delta\vu\cdot\nabla\h_1-\delta\h\dive\vu_2
    -\h_1\dive\delta\vu,\\
    G_2=&-\delta\vu\cdot\nabla\vu_1
    +2\mu\frac{\nabla\delta\h\dive\vu_2}{\hb+\h_2}
    +2\mu\frac{\nabla\h_1\dive\delta\vu}{\hb+\h_2}\\
    &+2\mu\left(\frac{1}{\hb+\h_2}-
    \frac{1}{\hb+\h_1}\right)\nabla\h_1\dive\vu_1\\
    &+2\mu\frac{\nabla\delta\h D(\vu_2)}{\hb+\h_2}
    +2\mu\frac{\nabla\h_1 D(\delta\vu)}{\hb+\h_2}\\
    &+2\mu\left(\frac{1}{\hb+\h_2}-
    \frac{1}{\hb+\h_1}\right)\nabla\h_1 D(\vu_1).
\end{align*}

Similar to \eqref{swm}, we can get
\begin{align*}
    \norm{(\delta\h,\delta\vu)}_{E_T^1}
    \ls
    Ce^{C\norm{\vu_2}_{L_T^1(\be[2])}}
    \Big(\norm{F_2}_{L_T^1(\hbe[0]{1})}+\norm{G_2}_{L_T^1(\be[0])}\Big).
\end{align*}
Noticing that $\h_1\in L_T^\infty(\hbe[0]{1})\cap L_T^1(\hbe[4]{3})$ and
$\vu_2\in L_T^1(\be[2])$, we can get
\begin{align*}
    \norm{F_2}_{L_T^1(\be[0])}
    \lesssim &  T^{\frac{1}{2}}\norm{\delta\vu}_{L_T^\infty(\be[0])}
    \norm{\h_1}_{L_T^2(\be[2])}
    +\norm{\delta\h}_{L_T^\infty(\be[0])}\norm{\vu_2}_{L_T^1(\be[2])}\\
    &+\norm{\h_1}_{L_T^\infty(\be[0])}\norm{\delta\vu}_{L_T^1(\be[2])}.
\end{align*}
Moreover, from $\h_1\in L_T^2(\be[2])\cap L_T^\infty(\be[1])$ by
Lemma~\ref{lem.inter}, we have
\begin{align*}
    \norm{F_2}_{L_T^1(\be[1])}
    \lesssim&
    \norm{\delta\vu}_{L_T^2(\be[1])}
    \norm{\h_1}_{L_T^2(\be[2])}\\
    &+\norm{\delta\h}_{L_T^\infty(\be[1])}
    \norm{\vu_2}_{L_T^1(\be[2])}
    +\norm{\h_1}_{L_T^\infty(\be[1])}
    \norm{\delta\vu}_{L_T^1(\be[2])}.
\end{align*}
Noting that $\h_1,\,\h_2\in L_T^\infty(\be[1])$, $\vu_1,\,\vu_2\in
L_T^1(\be[2])$, and
\begin{align*}
    \norm{\h_1}_{L^\infty([0,T]\times\Real^2)}\ls\frac{1}{2}\hb,\quad
\norm{\h_2}_{L^\infty([0,T]\times\Real^2)}\ls\frac{1}{2}\hb,
\end{align*}
by the construction of solutions, we have
\begin{align*}
    &\norm{G_2}_{L_T^1(\be[0])}\\
    \lesssim &\norm{\delta\vu}_{L_T^\infty(\be[0])}
    \norm{\vu_1}_{L_T^1(\be[2])}
    +4\mu(1+\norm{\h_2}_{L_T^\infty(\be[1])})
    \norm{\delta\h}_{L_T^\infty(\be[1])}
    \norm{\vu_2}_{L_T^1(\be[2])}\\
    &+4\mu(1+\norm{\h_2}_{L_T^\infty(\be[1])})
    \norm{\h_1}_{L_T^\infty(\be[1])}
    \norm{\delta\vu}_{L_T^1(\be[2])}\\
    &+4\mu
    \norm{\delta\h}_{L_T^\infty(\be[1])}
    \norm{\h_1}_{L_T^\infty(\be[1])}
    \norm{\vu_1}_{L_T^1(\be[2])}\\
    &\quad\times(1+\norm{\h_1}_{L_T^\infty(\be[1])}
    +\norm{\h_2}_{L_T^\infty(\be[1])}
    +\norm{\h_1}_{L_T^\infty(\be[1])}
    \norm{\h_2}_{L_T^\infty(\be[1])}).
\end{align*}
Thus, we obtain
\begin{align*}
    \norm{(\delta\h,\delta\vu)}_{E_T^1}
    \ls
    Ce^{C\norm{\vu_2}_{L_T^1(\be[2])}}\Big\{&
    \left(1+T^{\frac{1}{2}}+4\mu\hb^{-1}(1+\norm{\h_2}_{L_T^\infty(\hbe[0]{1})})\right)
    \\
    &\cdot\norm{\h_1}_{L_T^\infty(\hbe[0]{1})}
    +Z(T)
    \Big\}\norm{(\delta\h,\delta\vu)}_{E_T^1},
\end{align*}
where $\lim\sup_{T\to 0^+} Z(T)=0$.

Supposing that $\frac{A^2+A+2}{A(A+2)}(A+1)\te[C]E(0)<\frac{1}{4}$ besides \eqref{uniest7} for $E(0)$ and taking $0<T\ls 1$
small enough such that $C\norm{\vu_2}_{L_T^1(\be[2])}\ls \ln 2$ and
$Z(T)<\frac{1}{2}$, we obtain
$\norm{(\delta\h,\delta\vu)}_{E_T^1}\equiv 0$. Hence,
$(\h_1,\vu_1)\equiv (\h_2,\vu_2)$ on $[0,T]$.

Let $T_m$ (supposedly finite) be the largest time such that the two
solutions coincide on $[0,T_m]$. If we denote
\begin{align*}
    (\tilde{\h}_i(t),
    \tilde{\vu}_i(t)):=(\h_i(t+T_m),\vu_i(t+T_m)), \quad i=1,2,
\end{align*}
we can use the above arguments and the fact that
\begin{align*}
    \norm{\tilde{\h}_i}_{L^\infty(\Real^+\times\Real^2)}\ls
\frac{1}{2}\hb \quad \text{ and}\quad \norm{\tilde{\h}_i}_{L^\infty(\Real^+;
\be[0]\cap\be[1])} \ls A\te[C]E(0)
\end{align*}
to prove that
$(\tilde{\h}_1,\tilde{\vu}_1)=(\tilde{\h}_2,\tilde{\vu}_2)$ on the
interval $[0,T_m]$ with the same $T_m$ as in the above.
Therefore, we complete the proofs.

 \section*{Acknowledgments}
C.C. Hao was partially supported by the National Natural Science Foundation of China (NSFC) (grants
No. 10601061 and 10871134), the Scientific Research Startup Special Foundation for the Winner of the
Award for Excellent Doctoral Dissertation and the Prize of President Scholarship of Chinese Academy
of Sciences (CAS) and the Fields Frontier Project for Talented Youth of CAS. L. Hsiao was partially supported
by the NSFC (grant No. 10871134). H.-L. Li was partially supported
by the NSFC (grant No. 10871134),
the Beijing Nova program, the NCET support of the Ministry of Education of China, and the Huo Ying
Dong Foundation 111033.



\begin{thebibliography}{00}

\bibitem{Ber76} J. Bergh, J. L\"ofstr\"om,
{Interpolation Spaces, An Introduction}, Grundlehren der
mathematischen Wissenschaften 223, Springer-Verlag, Berlin
Heidelberg, 1976.

\bibitem{BreschDesjardins03b}  D. Bresch,   B. Desjardins,
 {Some diffusive capillary models for Korteweg type},
  C. R. Mecanique 331, 2003.

  \bibitem{BreschDesjardins03}  D. Bresch,   B. Desjardins,
 {Existence of global weak solutions for a 2D viscous
    shallow water equations and convergence to the quasi-geostrophic    model}, Comm. Math. Phys., 238 (2003), 211--223.

\bibitem{Bui81} A. T. Bui,  {Existence and uniqueness of a classical solution of
an initial boundary value problem of the theory of shallow waters},
SIAM J. Math. Anal., 12 (1981),  229--241.

\bibitem{CheM01}
  J.-Y. Chemin,  N. Masmoudi,  {About lifespan of regular solutions of
equations related to viscoelastic fluids}, SIAM J. Math.
Anal., {33}(2001),  84--112.

\bibitem{CheZ07}
 J.-Y. Chemin, P. Zhang, {On the global wellposedness to the 3-D
incompressible anisotropic Navier-Stokes equations}, Comm.
Math. Phys., {272} (2007),  529--566.

\bibitem{CM08}
 Q. Chen, C. Miao, Z. Zhang, {Well-posedness for the viscous shallow
water equations in critical spaces}, preprint, arXiv:math/0606081v1, 2006.

\bibitem{ChTa07}
 B. Cheng, E. Tadmor,  {Long time existence of smooth solutions for the rapidly rotating shallow-water and Euler equations}, preprint,
arXiv:0706.0758v1, 2007.

\bibitem{Dan00}R. Danchin,  {Global existence in critical spaces for
compressible Navier-Stokes equations}, Invent. Math.,
141 (2000),  579--614.

\bibitem{Dan01}  R. Danchin,
{Global existence in critical spaces for flows of compressible
viscous and heat-conductive gases}, Arch. Rational Mech.
Anal., 160 (2001), 1--39.

\bibitem{Dan05} R. Danchin,  {Fourier Analysis Methods for PDEs, (Lecture
Notes)}, November 14, 2005.

\bibitem{DanD01} R. Danchin, {Desjardins, B.: Existence of solutions
for compressible fluid models of Korteweg type}, Ann. Inst.
Henri Poincar\'e Anal. Nonl., 18(1) (2001), 97--133.

\bibitem{GerbeauPerthame01}   J. F. Gerbeau, B. Perthame,  {Derivation of viscous
Saint-Venant system for laminar shallow water; Numerical validation},
Disc. Cont. Dyn. Sys., Ser. B, 1 (2001),
 89--102.

\bibitem{GuoJiuXin2007} Z.-H. Guo,  Q.-S. Jiu, Z. Xin,
{Spherically symmetric isentropic compressible flows with
density-dependent viscosity coefficients}, SIAM J.
Math. Anal., 39 (2008), 1402--1427.


\bibitem{Has08} B. Haspot, {Cauchy problem for viscous shallow water
equations with a term of capillarity}, preprint,
arXiv:0803.1939v1, 2008.

\bibitem{klo85} P. E. Kloeden,  {Global existence of classical solutions in the
dissipative shallow water equations}, SIAM J. Math. Anal.,
16 (1985),  301--315.

\bibitem{LiLiXin2008}  H.-L. Li,  J. Li, Z. Xin,
{Vanishing of vacuum states and blow-up phenomena for the
compressible Navier-Stokes equations}, Comm. Math. Phys., 281 (2008),  401--444.


\bibitem{Mar07} F. Marche,   {Derivation of a new two-dimensional viscous shallow
water model with varying topography, bottom friction and capillary
effects}, European J. Mech. B/Fluids, 26 (2007),
49--63.


\bibitem{MatN80} A. Matsumura, T. Nishida,  {The initial value problem for the
equations of motion of viscous and heat-conductive gases}, J.
Math. Kyoto Univ., 20 (1980), 67--104.


\bibitem{MelletVasseur05} A. Mellet,  A. Vasseur,
  {On the barotropic compressible Navier-Stokes equations},
  Comm. Partial Diff. Eqns., 32(1-3) (2007),  431--452.

\bibitem{Pee76} J. Peetre,   {New thoughts on Besov spaces}, Duke
University Mathematical Series 1, Durham N. C., 1976.

\bibitem{Sun96}   L. Sundbye,  {Global existence for Dirichlet problem for the
viscous shallow water equations}, J. Math. Anal. Appl.,
202 (1996),  236--258.

\bibitem{Sun98} L. Sundbye,  {Global existence for the Cauchy problem for the
viscous shallow water equations}, Rocky Mountain J. Math.,
28 (1998), 1135--1152.

\bibitem{WanX05} W. K. Wang, C-J. Xu,  {The Cauchy problem for viscous shallow water equations}, Rev. Mat. Iberoamericana,
21 (2005), 1--24.
\end{thebibliography}
\end{document}